\documentclass[preprint,11pt,authoryear]{elsarticle}
\usepackage{enumitem}
\usepackage{amssymb}
\usepackage{eurosym}
\usepackage{color}
\usepackage{bm}
\usepackage{booktabs}
\usepackage{enumitem}
\usepackage{multirow}
\usepackage{longtable}
\usepackage{arydshln}
\usepackage{algorithmicx}
\usepackage[top=2.2cm,bottom=2.5cm,left=2.5cm,right=2.5cm]{geometry}
\usepackage{hyperref}
\usepackage{mathtools}
\usepackage{algorithm,algpseudocode}
\usepackage{subcaption}
\usepackage{tikz}
\usepackage{pgfplots}
\usetikzlibrary{patterns}
\usetikzlibrary{arrows}
\usepackage{amsmath}
\usepackage{etoolbox}
\usepackage{setspace}
\usepackage{etoolbox}
\usepackage{amsthm,amsmath,eurosym}

\newtheorem{theorem}{Theorem}
\newtheorem{proposition}{Proposition}
\newtheorem{example}{ {Example}}

\newcommand{\R}{\mathbb{R}}

\newcommand{\pt}{\text{ }\forall\text{ }}
\newcommand{\tq}{\text{ }:\text{ }}
\newcommand{\N}{\mathbb{N}}

\newcommand{\io}{[0,1]}

\usepackage{resizegather}

\pgfplotsset{compat=1.18} 

\journal{European Journal of Operational Research}

\usepackage{lineno, blindtext}

\makeatletter
\def\ps@pprintTitle{%
  \let\@oddhead\@empty
  \let\@evenhead\@empty
  \def\@oddfoot{\reset@font\hfil\thepage\hfil}
  \let\@evenfoot\@oddfoot
}
\makeatother

\begin{document}

\begin{frontmatter}
    \title{A hybrid approach for building fuzzy numbers based \\ on data and expert knowledge}
    \author[JAEN1]{Diego {\sc  Garc{\'i}a-Zamora}$^*$}\ead{dgzamora@ujaen.es}
    \author[CEGIST]{Jos{\'e} {\sc Rui~Figueira}}\ead{figueira@tecnico.ulisboa.pt }
     \author[INESC-ID]{Miguel {\sc Couceiro}}\ead{miguel.couceiro@inesc-id.pt }
    
    \address[JAEN1]{Department of Mathematics, University of Ja{\'e}n, 23071 Ja{\'e}n, Spain}
    \address[CEGIST]{CEGIST, Instituto Superior T\'{e}cnico,  Universidade de Lisboa, Portugal}
    \address[INESC-ID]{INESC-ID, Instituto Superior T\'{e}cnico,  Universidade de Lisboa, Portugal}
    \cortext[cor1]{Corresponding author at: Department of Mathematics, University of Ja{\'e}n, 23071 Ja{\'e}n, Spain}

    \begin{abstract}
    \noindent This paper presents a hybrid socio-technical methodology for constructing fuzzy numbers from numerical data while incorporating expert knowledge through an interactive Deck of Cards (DoC) process. The approach extends the existing DoC membership function construction framework by introducing a data-driven pipeline based on a convex version of fuzzy $k$-Means in which each computational step produces intermediate outputs that are translated into card-based structures for expert validation and tuning. The proposed method ensures interpretability, adaptability, and consistency between empirical evidence and expert semantics.
    \end{abstract}
    \vspace{0.25cm}
\end{frontmatter}
\section{Introduction}
The construction of appropriate membership functions is fundamental to the reliability and interpretability of fuzzy models \citep{DOMBI19901}. Over the past decades, two main research lines have emerged for determining membership functions \citep{Schwaab2015,Bilgiç2000}: data-driven approaches and expert-driven methodologies.

Data-driven methods rely on exploiting the information extracted from numerical datasets \citep{Yadav2014}. These techniques typically optimize membership parameters to reflect statistical patterns, leveraging clustering  \citep{dubois1980,DalleauCS20,Khairuddin2021,preudhomme}, neural networks \citep{Ross2010,George2020}, histogram-based identification \citep{DUBOIS198315,Soelistijanto2022}, 
statistical modelling \citep{Majumder2007,Porebski2016}, metaheuristics such as genetic algorithms \citep{Suryana2009,Khairuddin2022} or particle swarm 
optimization \citep{Fang2008,Li2019}, and general optimization schemes \citep{Wu2014,Bhattacharyya2020}. These contributions have been extensively adopted in classification, control, image processing, and pattern recognition problems \citep{Schwaab2015,MEDASANI1998391,Miliauskaite2020,HASUIKE2015994}. Nevertheless, their robustness completely depends on the quality and abundance of available data, limiting their applicability in contexts with scarce or 
heterogeneous observations \citep{Schwaab2015}.

Expert-driven methods pursue the opposite direction by constructing membership functions based on human judgment rather than empirical evidence. This area includes classical linguistic modelling frameworks \citep{2-tupleBook}, where preferences are elicited through intuitive procedures such as direct rating \citep{Norwich1984}, polling \citep{Hersh1976,AhmadShukri2021,Jain2010,Nguyen2014},
interval estimation \citep{Wang1986,Ukhobotov2017}, reverse rating \citep{Turksen1991}, exemplification \citep{Norwich1984}, or pairwise comparison \citep{Nieto-Morote2023}. These manual approaches enable the capture of subjective nuances in human cognition, but tend to be time-consuming and may require the assistance of an analyst to avoid cognitive overload for non-technical experts \citep{sancho1999}.

More recent solutions in the expert-driven line attempt to simplify the elicitation process by assigning predefined shapes (often triangular or trapezoidal) to the membership functions \citep{Wu2014,Miliauskaite2020}. Representative examples include the 2-tuple fuzzy linguistic modelling \citep{Herrera2000} and Hesitant Fuzzy Linguistic Term Sets \citep{Rodriguez2012}, which remain highly influential in the literature.  However, such predefined structures implicitly assume a shared semantic understanding among decision-makers, neglecting the variability in how individuals interpret qualitative terms \citep{MG1,HET1}. This limitation becomes even more pronounced in group decision-making settings, where information is expressed using different scales \citep{HET1,HET2}, with different semantic perceptions of language \citep{PIS1,PIS2}, and with different levels of granularity \citep{MG1,MG2}.

In this context, there is a research gap regarding the construction of membership functions that accounts simultaneously for empirical patterns and human interpretability \citep{DOMBI2022206,Medvediev2020,Kara2021,Wang2022}. Most existing approaches only involve experts at a preliminary stage and do not allow them to actively co-construct the membership functions as the model evolves \citep{Chakraborty2001}. Recent works stress the need for socio-technical frameworks in which decision-makers and analysts jointly build uncertainty representations through iterative feedback and negotiation \citep{corrente2021}.

The \emph{Deck-of-Cards-based Membership Functions} (DoC-MF) methodology, recently introduced in \cite{DoC-MF}, provides a structured and interpretable approach for constructing fuzzy numbers through expert elicitation. By enabling decision-makers to express qualitative differences using an intuitive card-based representation, the method successfully bridges human reasoning and uncertainty modelling without requiring direct numerical judgments.

However, the original DoC-MF framework was designed mainly for expert-driven contexts, where membership functions must be constructed solely by interacting with a decision-maker. In many real-world applications, relevant numerical datasets are available and could contribute to the modelling of uncertainty. Sole reliance on expert assessments may lead to biases or representations inconsistent with empirical evidence.

To address these limitations, we propose a hybrid socio-technical methodology that integrates data-driven insights with the interpretability and flexibility of the DoC-MF approach. The main contributions of this work are threefold:
\begin{itemize}
    \item We introduce a novel data-driven pipeline that extracts preliminary fuzzy information from numerical observations given as frequency tables.
    \item We define a mechanism to translate intermediate computational outputs into the card-based structures from the DoC-MF approach, enabling expert interaction, validation, and refinement at every stage.
    \item We provide a unified membership-function construction process that ensures consistency between empirical patterns and expert semantics.
\end{itemize}

To support the data-driven component, we also develop a modified version of the classical fuzzy $k$-means clustering algorithm, referred to as \emph{convex FKM} (C-FKM). By enforcing convexity constraints, the method guarantees that all clusters produced during computation correspond to valid fuzzy numbers compatible with the DoC elicitation protocol.

Thus, the resulting hybrid approach preserves transparency and interpretability while adapting to heterogeneous or non-standard data distributions. By integrating empirical data with expert intuition to construct membership functions, this hybrid approach addresses critical challenges across diverse high-stakes domains. In predictive medicine, for instance, it enables the refinement of diagnostic models by reconciling quantitative clinical metrics with the nuanced qualitative judgment of practitioners. Similarly, in industrial process control and autonomous systems, this integration allows automated controllers to inherit the resilience and 'common sense' of human operators, ensuring stability even when environmental conditions deviate from historical patterns. Furthermore, in fields such as environmental risk assessment or financial forecasting, the capacity to anchor mathematical models in subjective expertise ensures a more faithful representation of reality, overcoming the limitations of purely data-driven models that may lack context or fail to account for rare but impactful events. In this manuscript, we focus on a numerical study on real educational performance data that demonstrates the ability of our approach to produce semantically meaningful membership functions aligned with the underlying evidence.

The remainder of this paper is as follows. Section \ref{sec:preliminaries} introduces the basic notions in the literature and the notation necessary to understand the paper. In Section \ref{sec:convex_fkm}, we develop the theoretical basis to extend the classical fuzzy $k$-means to be compatible with the DoC-MF elicitation method. Subsequently, Section \ref{sec:hybrid} describes in detail our hybrid methodology based on data and expert interaction to obtain fuzzy numbers. In Section \ref{sec:numerical_example}, we show a complete example of the hybrid process in a concrete situation, whereas in Section \ref{sec:comparison} we provide some numerical comparisons to illustrate the performance of the method when facing different data distributions. Finally, Section \ref{sec:conclusion} concludes the manuscript.

\section{Preliminaries}
\label{sec:preliminaries}


In this section we recall the basic terminology and notation needed throughout the paper. We will also revisit  the ``Deck-of-Cards method'' and ``Fuzzy $k$-means clustering'' that will be the bases of the fuzzy number construction that we propose. 

\subsection{Fuzzy sets and fuzzy numbers}
 A fuzzy set on $\R$ is a mapping $A:\R\to[0,1]$. For $\alpha\in]0,1]$, the $\alpha$-cut of $A$ is the subset $A^\alpha=\{x\in\R\tq A(x)\geq \alpha\}$ whereas $A^0$ denotes the topological closure of the support of $A$ \cite{wang2012mathematics}. Note that $A^0$, which we will call the support of $A$, models the points of the real line that, to some extent, show some belongingness to the fuzzy set $A$, whereas $A^1$, the so-called core of $A$, denotes the values in $\R$ for which we have full certainty of their belongingness to $A$.

A fuzzy number is a fuzzy set ${A}$  satisfying (i) $A^1\neq \emptyset$ ({\it normality}), (ii) $A^\alpha$ are intervals for any $\alpha\in\io$ ({\it convexity}), and (iii) $A^0$ is bounded. Fuzzy numbers express the degree to which a value $x\in\R$ is considered compatible with the corresponding uncertainty concept. For the representation to be meaningful in decision contexts, the membership function must satisfy several structural properties that ensure interpretability and internal consistency.
Equivalently, the membership function increases at the left-hand side of the core and decreases at its right-hand side. This left–right unimodal behaviour reflects the assumption that uncertainty is highest in the tails and lowest around the central concept represented by the core. 

Finally, a family of $k\in\N$ fuzzy numbers $A_1,\ldots ,A_k$ is called a fuzzy partition over $[a,b]$ if $\sum_{j=1}^kA_j(x)=1$, for all $x\in[a,b]$. This notion is essential in this work because it ensures that all membership functions collectively represent the whole domain without ambiguity, allowing each value to express its affiliation to different concepts while maintaining global interpretability and consistency in decision-making.

\subsection{The Deck-of-Cards method}

The DoC method is an elicitation technique originally designed to assist experts in expressing preferences and qualitative distinctions using a simple and cognitively intuitive representation \cite{corrente2021}. Instead of numerical scales, 
decision makers manipulate a collection of identical cards to indicate 
relative differences between ordered categories. The number of blank cards inserted between two reference points reflects the perceived strength of the distinction: the more cards-units placed in between, the more distant the concepts are considered to be. For instance, let us assume that there are three ordered objects $l_1\prec l_2\prec l_3$ whose performance has to be assessed. Suppose that the decision-maker establishes that the difference in intensity of his/her preference between the first and second objects can be modeled with four units, whereas his/her preference between the second and third objects can be modeled with six, i.e.: 
\begin{equation*}
    l_1 \quad [4] \quad l_2 \quad [6] \quad l_3
\end{equation*}
Then, the analyst assigns numerical values in a $[0,1]$ scale to these objects proportionally:
\begin{equation*}
    v(l_1)=0\quad v(l_2)=0.4 \quad v(l_3)=1.
\end{equation*}

This approach avoids imposing direct numerical 
judgements on experts, while preserving ordinal and relational information. A key property of the DoC method is that it can approximate any set of values, using a large enough number of cards, as stated in Theorem~\ref{th: num_to_Card}. 
\begin{theorem}[\cite{DoC-MFT2}]\label{th: num_to_Card}
Let ${\bf x}=(x_0,\ldots,x_n)\in[0,1]^{\,n+1}$ with $n>1$, be an $n$-tuple satisfying 
$$0=x_0 < x_1 < \cdots < x_n = 1,$$ and let $m\in\mathbb{N}$ be such that $
\big\lfloor 10^m x_{i-1} \big\rfloor 
<
\big\lfloor 10^m x_i \big\rfloor,~
 i=1,\ldots,n.$ Define the rational approximation
\[
r_i=\frac{\lfloor 10^m x_i \rfloor}{10^m},
\quad i=0,\ldots,n,
\]
with $N = 10^m$ being the total number of units.
Then, the DoC method can represent the ordered tuple $x$ with precision $10^{-m}$, using a sequence of $N$ cards partitioned into $n$ groups of consecutive units determined recursively by
\[
\begin{aligned}
c_1 &= N r_1 = \lfloor 10^m x_1 \rfloor, \\[1mm]
c_i &= N r_i - \sum_{j=1}^{i-1} c_j 
     = \lfloor 10^m x_i \rfloor - \sum_{j=1}^{i-1} c_j,
     \qquad i = 2,\ldots,n.
\end{aligned}
\]
The integers $c_1,\ldots,c_n$ specify exactly the number of units to place between $r_{i-1}$ and $r_i$, and satisfy
\[
r_i = \frac{1}{N} \sum_{j=1}^i c_j,
\qquad i=1,\ldots,n.
\]
Hence, the DoC representation constructed from these recursive counts approximates the original tuple $x$ with accuracy $10^{-m}$.
\end{theorem}

In recent work, the DoC methodology has been extended to support the 
construction of fuzzy numbers \cite{DoC-MF} to obtain a fuzzy representation of the linguistic terms of a given scale. To do so, a three-step socio-technical methodology is carried out. First, a value scale for the representative points of the fuzzy numbers is identified by placing blank cards between the different levels of the scale. Each card corresponds to a discrete unit, and the cumulative structure of inserted cards can be naturally translated into values within $[0,1]$. Subsequently, the boundaries of the support and the core are inferred by questioning the decision-maker about the range of absence or full confidence for the corresponding fuzzy concept. Finally, the DoC method is used again to refine the left and right-hand sides of the membership functions.

This DoC-based elicitation process provides several advantages. First, it 
avoids arbitrary functional assumptions such as triangular or trapezoidal 
profiles, letting the membership curve emerge from expert reasoning. Second, 
the card representation allows decision makers to revise and justify modelling 
choices incrementally, facilitating communication and negotiation in 
multi-stakeholder settings. In addition, the approach establishes traceability 
between uncertainty modelling and human judgement, which is crucial for 
applications in which acceptance of the results depends on interpretability.

\subsection{ Fuzzy k-means clustering}
\label{subsec:fkm}

Since several components of the proposed hybrid approach rely on the
partitioning of empirical data into graded groups, we recall here the
classical  { FKM} (also known as  {fuzzy $c$-means})
algorithm and introduce the notation used throughout this work.

Let $D=\{x_1,\ldots,x_n\}$ be a set of real-valued observations.
Given a fixed number $k \ge 2$ of clusters, the goal of FKM
is to determine both the cluster centers (prototypes) $v_1,\ldots,v_k \in \mathbb{R}$, and the membership degrees $u_{ij} \in [0,1]$ expressing the degree to which $x_i$ belongs to cluster $j$,  and such that for each data point $x_i, ~i=1,\ldots,n,$ we have that $\sum_{j=1}^k u_{ij} = 1$, and that the memberships are soft rather than crisp. The classical objective function is $J\colon [a,b]^{nk}\times[a,b]^k$  defined by 
\begin{equation*}
J(U,V) = \sum_{i=1}^n \sum_{j=1}^k u_{ij}^m\, \|x_i - v_j\|^2,
\label{eq:fkm_obj}
\end{equation*}
where $U = (u_{ij})$ is the membership matrix, 
$V = (v_1,\ldots,v_k)$ is the tuple of cluster centers, and $m>1$ is the standard fuzzifier parameter (controlling the degree of fuzziness); for further background see \cite{BEZDEK1984191}. 
Typical values are $m=1.5$ or $m=2$. The minimization of $J$ under the membership constraints leads to the well-known update equations:
\begin{equation}
\label{eq:u&v}
u_{ij}
=
\left(
\sum_{\ell=1}^k 
\left(
    \frac{\|x_i - v_j\|}{\|x_i - v_\ell\|}
\right)^{\!\frac{2}{m-1}}
\right)^{-1}, \quad \text{where}
\quad
v_j
=
\frac{
    \sum_{i=1}^n u_{ij}^m\, x_i
}{
    \sum_{i=1}^n u_{ij}^m
},
\quad j=1,\ldots,k.
\end{equation}
In the classical approach, these two equations are applied iteratively until convergence \cite{BEZDEK1984191}, typically until
\(
\|V^{(t+1)} - V^{(t)}\|
\)
falls below a tolerance threshold. 

\section{Convex  fuzzy $k$-means}
\label{sec:convex_fkm}

In this section, we introduce C-FKM, an adaptation of the classical FKM specifically designed so that every class produced by the algorithm is a  {fuzzy number} by construction. The method enforces a local (convex) support around each centroid and forces the fuzzy partition to be adjacent in the value axis: each observation may belong only to the two clusters whose centres bracket the observation. This locality constraint yields membership functions with compact supports and non-empty cores, which are essential for the hybrid construction of interpretable membership functions that represent value scales.

In the following, we describe the model, give closed-form formulae to compute centres and memberships under the new constraints, present a convergence result, and finish with a simple iterative algorithm to compute a C-FKM solution.

Let $D=\{x_1,\dots,x_n\}\subset[a,b]\subset\mathbb R$ be the dataset, in which $[a,b]$ are bounds chosen either from the data or from an external source. Let us fix the number of clusters $k\ge 2$ and the fuzzifier parameter $m>1$. Denote by $V=(v_1,\dots,v_k)$ the ordered centroids with $a < v_1 < v_2 < \cdots < v_k < b.$ For each observation $x_i$, we define the  {left-bracketing index} $j(i)$ as the unique index such that $v_{j(i)} \le x_i < v_{j(i)+1},$ with the convention $v_0:=a$ and $v_{k+1}:=b$. Thus $x_i$ lies in the interval $[v_{j(i)},v_{j(i)+1})$. Let us denote by $U=(u_{ij})\in\mathcal{M}_{n\times k}([0,1])$ the membership matrix obeying the standard simplex constraints
\[
u_{ij}\ge 0,\qquad \sum_{j=1}^k u_{ij}=1,\qquad i=1,\dots,n.
\]
The key  {adjacency constraint} of C-FKM is that for each $i$ we allow nonzero memberships only to the two adjacent clusters:
\[
u_{ij}=0\quad\text{for all } j\notin\{j(i),\, j(i)+1\}.
\]
Finally, as in classical  FKM, we consider the objective function $J:\mathcal{M}_{n\times k}([0,1])\times[a,b]^k$
\[
J(U,V) \;=\; \sum_{i=1}^n \sum_{j=1}^k u_{ij}^m\, (x_i-v_j)^2, \quad\pt U\in \mathcal{M}_{n\times k}([0,1]), V\in[a,b]^k
\]
Therefore, the C-FKM problem is the constrained minimization of $J$ subject to the simplex and adjacency constraints above:
\[
\begin{aligned}
\min_{\substack{U=(u_{ij})\\ V=(v_1,\dots,v_k)}} 
\; J(U,V)
&= \sum_{i=1}^n \sum_{j=1}^k u_{ij}^{\,m}\, (x_i-v_j)^2 \\[4pt]
\text{subject to}\qquad 
& u_{ij} \ge 0,\qquad 
\sum_{j=1}^k u_{ij}=1,\qquad i=1,\dots,n, \\[4pt]
& u_{ij} = 0 \quad \text{for all } j\notin\{j(i),\, j(i)+1\}, \\[4pt]
& a < v_1 < v_2 < \cdots < v_k < b.
\end{aligned}
\]

Note that the restriction that each $x_i$ can only belong to the two adjacent clusters is the structural change that enforces convex and contiguous supports, and guarantees fuzzy-number outputs, as we will prove below. However, following the structure of the classical FKM, let us start with two theorems that give closed-form updates in the two alternating steps of the algorithm: (i) compute centers given memberships, and (ii) compute memberships given centers.

\begin{proposition}[Center update]\label{th:update_centers}
Let $U=(u_{ij})\in\mathcal{M}_{n\times k}([0,1])$ be any feasible membership matrix satisfying the simplex and adjacency constraints. Then the minimizer $V^\star$ of $J_U:[a,b]^k\to\R$, {\it i.e.,}  the function $V\to J(U,V) $ resulting of keeping $U$ fixed in $J$, is unique and given by 
\begin{equation*}\label{eq:center_update}
v_j^\star \;=\; \frac{\displaystyle\sum_{i=1}^n u_{ij}^m\, x_i}
                     {\displaystyle\sum_{i=1}^n u_{ij}^m},
\qquad j=1,\dots,k.
\end{equation*}
\end{proposition}

\begin{proof}
For fixed $U\in\mathcal{M}_{n\times k}([0,1])$, the objective is given by $$J_U(V)=\sum_{j=1}^k \sum_{i=1}^n u_{ij}^m\, (x_i-v_j)^2, ~\text{for} ~V\in[a,b]^k.$$ 
Thus  the global minimum is unique and can be found when the gradient $ \nabla J_U$ is null, {\it {\it i.e.,} }  \begin{equation}\label{eq:gradient}
   \nabla J_U=[\frac{\partial}{\partial v_1}J_U(V),\ldots, \frac{\partial}{\partial v_k}J_U(V)]=\mathbf{0}=[0,\dots,0]. 
\end{equation}
Since $\frac{\partial}{\partial v_j}J_U(V)=
-2\sum_{i=1}^n u_{ij}^m (x_i - v_j)$, it then follows that \eqref{eq:gradient} holds only if $
v_j = \frac{\sum_{i=1}^n u_{ij}^m x_i}{\sum_{i=1}^n u_{ij}^m},$ and s the prooof is now complete.
\end{proof}

We now provide a way to compute memberships given the cluster centers.
\begin{theorem}[Membership update]\label{th:update_MFs}
Let us consider fixed centroids $V=(v_1,\dots,v_k)$ with $a< v_1<\cdots<v_k<b$. For each $i\in\{1,\dots,n\}$, let $j=j(i)$ be the left-bracketing index so that $v_j\le x_i < v_{j+1}$. Under the adjacency constraint that $u_{i\ell}=0$ for $\ell\notin\{j,j+1\}$, the unique minimum $U^\star\in\mathcal{M}_{n\times k}([0,1])$ of $J_V(U)=\sum_{j=1}^k u_{ij}^{m}d_{ij}$, where $d_{ij}=(x_i-v_j)^2\pt i=1,...,n,j=1,...,k$, subject to $\sum_\ell u_{i\ell}=1$ and $u_{i\ell}\ge 0$ is given by 
\begin{enumerate}
    \item $u_{i,j} = \Big(1 + \Big(\dfrac{d_{i,j}}{d_{i,j+1}}\Big)^{\!1/(m-1)}\Big)^{-1},$ and $
u_{i,j+1} =\Big( 1+\Big(\dfrac{d_{i,j+1}}{d_{i,j}}\Big)^{\!1/(m-1)}\Big)^{-1}, $ whenever $d_{i,j}>0$ and $d_{i,j+1}>0$, 
\item  $u_{i,j}=1$ and $u_{i,j+1}=0$, whenever  $d_{i,j}=0$ and $d_{i,j+1}>0$, and  
\item $u_{i,j+1}=1$ and $u_{i,j}=0$, whenever $d_{i,j+1}=0$ and $d_{i,j}>0$.
\end{enumerate} 
\end{theorem}

\begin{proof}
Given the adjacency constraint, for a fixed $i=1,...,n$, the objective function reduces
to
\[
J_i(u_{i,j},u_{i,j+1}) = u_{i,j}^m d_{i,j} + u_{i,j+1}^m d_{i,j+1},
\]
subject to the constraints $ u_{i,j}+u_{i,j+1}=1,\ u_{i,j},u_{i,j+1}\ge 0$. Therefore, we can eliminate $u_{i,j+1}=1-u_{i,j}$ and minimize the one-dimensional convex
function $f(t)=t^m d_{i,j} + (1-t)^m d_{i,j+1}$ on $t\in[0,1]$. For the
nondegenerate case $d_{i,j},d_{i,j+1}>0$ the first-order condition
$f'(t)=0$ yields
\[
m t^{m-1} d_{i,j} - m (1-t)^{m-1} d_{i,j+1} = 0
\]
and therefore
\[
\Big(\frac{t}{1-t}\Big)^{m-1} = \frac{d_{i,j+1}}{d_{i,j}}
\quad\Longrightarrow\quad
\frac{t}{1-t} = \Big(\frac{d_{i,j+1}}{d_{i,j}}\Big)^{1/(m-1)}.
\]
Solving it for $t$, gives exactly the expression in the statement.
The degenerate cases where one distance is zero follow from the observation that $J_i\ge 0$ and any zero-distance term forces the objective to be minimized by placing full weight on that zero-distance cluster.
\end{proof}

\begin{theorem}\label{th:CFKM_fuzzynumbers}
Let $V=(v_1,\dots,v_k)$ be a set of ordered centroids, let $U=(u_{ij})$ be the membership matrix obtained in Theorem \ref{th:update_MFs}, with fuzzifier $m>1$, and let $V_j$ denote the membership function of cluster $j=1,...,k$ obtained by interpolation of the values $u_{ij}$ over $x$ considering the centroids $V$. Then, for every $j=1,\dots,k$ the function $V_j$ satisfies
\begin{itemize}
  \item normality ($V_j(v_j)=1$),
  \item its support is contained in $[v_{j-1},v_{j+1}]$,
  \item it is monotone increasing on the left interval $[v_{j-1},v_j]$ and
        monotone decreasing on the right interval $[v_j,v_{j+1}]$.
\end{itemize}
Hence, each $V_j$ is a fuzzy number, and the collection
$\{V_1,\dots,V_k\}$ is a fuzzy partition of $[a,b]$.
\end{theorem}

\begin{proof}
Normality and the support condition follow immediately from the model
construction and the adjacency constraint. Indeed, by construction,
in the minimization for a data point located exactly at $v_j$, the unique optimal assignment (see Theorem~\ref{th:update_MFs}) gives full membership to cluster $j$, hence $V_j(v_j)=1$ (normality). The adjacency constraint forces any data point $x$ to receive positive membership for cluster $j$ only if $x\in[v_{j-1},v_{j+1})$, so the support of $V_j$ is contained in the compact interval $[v_{j-1},v_{j+1}]$.

It remains to show the monotonicity assertions. Fix an index $j$ and
consider first the right-hand side interval $[v_{j},v_{j+1}]$. Let us consider an index $i=1,...,n$ such that $v_{j}<x_i<x_{i+1}<v_{j+1}$. In that case, the monotonicity of $x\to (x-v_j)^2$ and $x\to (x-v_{j+1})^2$  implies that $d_{i,j}d_{i+1,j+1}\leq d_{i+1,j}d_{i,j+1}$. If we assume that all these values are positive,  we obtain
\begin{align*}
    \dfrac{d_{i+1,j}}{d_{i+1,j+1}}\geq \dfrac{d_{i,j}}{d_{i,j+1}} &\iff 
    1 + \Big(\dfrac{d_{i+1,j}}{d_{i+1,j+1}}\Big)^{\!1/(m-1)}\geq 1 + \Big(\dfrac{d_{i,j}}{d_{i,j+1}}\Big)^{\!1/(m-1)}\\
    &\iff \Big(1 + \Big(\dfrac{d_{i,j}}{d_{i,j+1}}\Big)^{\!1/(m-1)}\Big)^{-1}\geq \Big(1 + \Big(\dfrac{d_{i+1,j}}{d_{i+1,j+1}}\Big)^{\!1/(m-1)}\Big)^{-1}\\
    &\iff  u_{ij}\geq u_{(i+1)j},
\end{align*}
which is the monotonicity of $V_j$ in its right-hand side. In the case that some of them are zero, the monotonicity still holds, taking into account that some memberships will be equal to $1$. We omit here the complete discussion for the sake of space. The argument for the left-hand side $[v_j,v_{j+1}]$ is analogous.

Combining normality, compact support, and the established monotonicity on both sides shows that each $V_j$ is a normal, convex membership function supported on a compact interval, which is precisely the definition of a fuzzy number. Finally, since for each $x_i$, the adjacency constraint enforces that the memberships across clusters sum to one, the family $\{\mu_j\}$ forms a fuzzy partition of $[a,b]$. This completes the proof.
\end{proof}

The following theorem justifies the alternating optimization that we will use in practice: update memberships with Proposition \ref{th:update_MFs}, then centers with Proposition \ref{th:update_centers}, and repeat until convergence.

\begin{theorem}[Convergence to local minimum]
Let $\{(U^{(t)},V^{(t)})\}_{t\ge 1}$ be the sequence produced by the
C-FKM alternating updates (from a certain initialized centroid vector $V^0$):
\[
U^{(t+1)} \leftarrow \arg\min_U J(U,V^{(t)}) \quad\text{(using Proposition \ref{th:update_MFs})},
\]
\[
V^{(t+1)} \leftarrow \arg\min_V J(U^{(t+1)},V) \quad\text{(using Proposition \ref{th:update_centers})}.
\]
Then the sequence of objective values $J(U^{(t)},V^{(t)})$ is
nonincreasing and converges to a finite limit.
\end{theorem}

\begin{proof}
Note that each update step minimizes $J$ over a subset of the variables while keeping the others fixed. Consequently, \[
J(U^{(t+1)},V^{(t)}) \le J(U^{(t)},V^{(t)}),\qquad
J(U^{(t+1)},V^{(t+1)}) \le J(U^{(t+1)},V^{(t)}).
\]
Thus, the sequence of objective values is nonincreasing and bounded
below by $0$, and it must be convergent by the Monotone Convergence Theorem \cite{Spivak2008}.
\end{proof}

Note that the C-FKM algorithm alternates the closed-form updates just derived. Below, we give a concise pseudo-code that can be used in implementations.
\begin{enumerate}
    \item {Initialization:} choose $k$ initial centroids $v_1^{(0)}<\cdots<v_k^{(0)}$ (for instance by evenly spaced values between $a$ and $b$ or using percentiles). Set $t\leftarrow 0$.
    \item {Repeat:}\begin{enumerate}
        \item For each $i=1,\dots,n$, compute $j(i)$ (left-bracketing index) with respect to $V^{(t)}$. Compute the two distances          $d_{i,j(i)}$ and $d_{i,j(i)+1}$ and update memberships      $u_{i,j(i)}^{(t+1)},u_{i,j(i)+1}^{(t+1)}$ using Proposition~\ref{th:update_MFs}. Set all other $u_{i\ell}^{(t+1)}=0$.
        \item Update each centroid $v_j^{(t+1)}$ according to
              Proposition~\ref{th:update_centers} using $U^{(t+1)}$.
        \item If $\|V^{(t+1)}-V^{(t)}\|$ (or the corresponding decrease of $J$) is below a tolerance level $\tau>0$, stop; else set $t\leftarrow t+1$.
    \end{enumerate}
\end{enumerate}

Let us recall here the main advantages of the C-FKM algorithm. On the one hand, the adjacency constraint ensures that each membership function $V_j$ is zero outside the compact interval between neighboring centroids and also convex. Consequently, the resulting $V_j$ is a fuzzy number. On the other hand, the algorithm produces piece-wise linear fuzzy numbers, which can be integrated within the DoC-MF framework. In this sense, all the steps can be directly translated into cards and shown to the decision-makers for final refinement, as we show in the following section.

\section{Hybrid construction of fuzzy numbers integrating data and expert knowledge}\label{sec:hybrid}

In this section, we introduce a hybrid methodology for constructing 
fuzzy numbers from numerical data while incorporating expert knowledge 
through the DoC-MF method. The proposed process follows the same three-step structure of the DoC-MF approach (value scale construction, identification of core and support, and definition of the left-hand and right-hand sides) but extends each step with a preliminary data-driven procedure based on C-FKM. At every stage, the output produced by the data-driven computation is 
translated into a card-based representation that allows the decision-maker 
to adjust, validate, or refine the membership function before the method continues to the next stage. Below, we describe each phase in detail.

\subsection{Step 1: Value scale construction}
\label{subsec:step1}

The first stage of the hybrid procedure aims at constructing an initial value scale that reflects the empirical structure of the data while remaining interpretable for the decision-makers involved in the process. Following the spirit of the socio-technical approach introduced in the DoC-MF methodology \cite{DoC-MF}, this stage combines a purely data-driven component with an interactive expert-based refinement. Thus, the data-driven component extracts representative reference points from the dataset, whereas the expert component interprets and adjusts these reference points through the DoC method. 

Given the dataset $D=\{x_1,\ldots,x_n\}$, we begin by applying C-FKM (see Section~\ref{sec:convex_fkm}) to partition the data into $k$ overlapping groups. Only the cluster centroids produced by the algorithm, denoted by $v_1,\ldots,v_k$, are required in this step. These centroids serve as representatives of the data distribution, capturing areas of density and patterns in the numeric observations. To ensure consistency in the later stages of the method, we order the centroids so that 
\[
a < v_1 < v_2 < \cdots < v_k < b,
\]
where the bounds $[a,b]$ can be either obtained from the data or given by the decision-makers. Let us define the $(k+2)$-tuple
\[
v = (a, v_1, v_2, \ldots, v_k, b).
\]

Subsequently, the ordered tuple $v$ is translated into a DoC structure by means of Theorem~\ref{th: num_to_Card}, which guarantees that, after choosing an appropriate precision level $10^{-m}$, the distances between consecutive values can be translated into real cards that represent the differences of intensities between the original values. Specifically, for the chosen $m$, we compute the rational approximations $r_i=\lfloor 10^m v_i\rfloor /10^m$ and the recursive sequence of card
counts
\[
c_1 = \lfloor 10^m v_1\rfloor, ~\text{ and }
~
c_i = \lfloor 10^m v_i\rfloor - \sum_{j=1}^{i-1}c_j, ~\text{ for }
 i=2,\ldots,k+1,
\]
where each quantity $c_i$ corresponds to the number of units between
labels $v_{i-1}$ and $v_i$. Then, these units are physically represented as a card chain that translates the spacing of the
centroids, and therefore the geometric structure of the dataset, into an
interpretable visual model that the decision-makers can conveniently examine and manipulate.

Once the card chain is produced, it is presented to the decision-makers so that they can modify the spacing between reference values. This constitutes a crucial part of the hybrid methodology: even if the data suggest certain distances between centroids, the perceived semantic differences between adjacent values may differ from the purely numerical ones.  By inserting or removing blank cards between levels, decision-makers can increase or decrease the discriminatory power in specific regions of the scale. The expert-refined card configuration replaces the raw data-driven structure and becomes the definitive value scale for the subsequent steps of the process.
 
\begin{example}
    To illustrate this mechanism, let us consider a dataset scaled to $[0,1]$ and suppose that  FKM identifies three centroids at $v_1=0.18$, $v_2=0.43$, $v_3=0.72$. Including the endpoints ($0$ and $1$), we obtain the ordered tuple $x = (0,\; 0.18,\; 0.43,\; 0.72,\; 1).$ Assume a desired precision of $10^{-2}$, so that $m=2$ and $N=10^2=100$ units are available. Applying Theorem~\ref{th: num_to_Card}, we compute:
\[
\begin{aligned}
c_1 &= \lfloor 100\cdot 0.18 \rfloor = 18,\\
c_2 &= \lfloor 100\cdot 0.43 \rfloor - 18 = 25,\\
c_3 &= \lfloor 100\cdot 0.72 \rfloor - (18+25) = 29,\\
c_4 &= \lfloor 100\cdot 1.00 \rfloor - (18+25+29) = 28.
\end{aligned}
\]
Thus, the data-driven card chain contains:
\[
18 \text{ cards between } 0 \text{ and } 0.18,\qquad
25 \text{ cards between } 0.18 \text{ and } 0.43,
\]
\[
29 \text{ cards between } 0.43 \text{ and } 0.72,\qquad
28 \text{ cards between } 0.72 \text{ and } 1.
\]

When presented with this structure, a hypothetical decision-maker may judge that the difference between $0.18$ and $0.43$ is too small relative to the
semantic jump perceived between these levels (for example, if these
values correspond to distinct linguistic labels). The decision-maker may therefore insert 5 additional blank cards in the second
interval, bringing its total from 25 to 30. Conversely, the decision-maker may feel that the last interval, between $0.72$ and $1$, is overly stretched and removes 4 cards to reduce it from 28 to 24. The adjusted chain then replaces the original one and encodes both the empirical information contained in the centroids and the semantic judgment of the decision-maker regarding the spacing of values.
\end{example}

This refined value scale becomes the foundation upon which the core and support identification and the construction of the left and right sides of the membership function will be carried out in the subsequent steps of the
hybrid methodology.

\subsection{Step 2: Identification of Cores and Supports}
\label{subsec:step2}

Once the centroids have been validated by the decision-makers in Step~1, the next goal is to further analyze the cores and supports of the fuzzy numbers associated with each class. This requires transforming the validated centroids into complete membership functions, identifying their cores and supports, and allowing the decision-makers to refine the resulting intervals by interacting with the Deck of Cards representation.

Let $V=(v_1,\dots,v_k)$ denote the validated centroids. Using Theorem~\ref{th:CFKM_fuzzynumbers}, the membership function of the cluster associated to the centroid $v_j$ can be updated by the expressions
\begin{gather*}
u_{ij}=V_j(x_i)=
\begin{cases}
\displaystyle 
\left(
1+\left(\frac{(x_i-v_j)^2}{(x_i-v_{j-1})^2}\right)^{\!\frac{1}{m-1}}
\right)^{-1}, & \text{if } v_{j-1} \le x_i < v_{j}, \\[10pt]
\displaystyle 
\left(
1+\left(\frac{(x_i-v_j)^2}{(x_i-v_{j+1})^2}\right)^{\!\frac{1}{m-1}}
\right)^{-1}, & \text{if } v_j \le x_i < v_{j+1}, \\[10pt]
1,& \text{if } j=0 \text{ and }x_i\leq v_0, \text{ or }j=k \text{ and }x_i\geq v_k,\text{ or } x_i=v_j \\[10pt]
0, & \text{otherwise},
\end{cases}
\end{gather*}

Keep in mind that this guarantees (i) normality at $v_j$, (ii) compact support on $[v_{j-1},v_{j+1}]$, and (iii) monotonicity on each side of the centroid (see Theorem \ref{th:CFKM_fuzzynumbers}). To obtain interpretable fuzzy numbers, we define the  {core} of each class $V_j$ as
\[
\mathrm{Core}(V_j)
=
\{ x\in D : V_j(x) \ge 1-\tau \},
\qquad\text{with $\tau\in[0,1]$ typically set to } 0.01.
\]
Since $\mu_j$ is monotone on both sides of $V_j$, the core is always an interval $\mathrm{Core}(V_j) = [\,\underline{c}_j, \overline{c}_j\,]$. These thus-obtained bounds 
\begin{gather*}
    a=\underline{c}_1<\overline{c}_1<\underline{c}_2<\overline{c}_2<...<\underline{c}_k<\overline{c}_k=b
\end{gather*}
are converted into Deck of Cards units using Theorem~\ref{th: num_to_Card}.  
The resulting cards are then shown to the decision-makers, who may modify the cards to enlarge or shrink the cores. Note that this adjustment respects the constraint that cores of different fuzzy numbers cannot overlap. Furthermore, the validated centroids in the previous steps must lie within the core. To guarantee this, it is possible to include the validated centroids in the former chain of values and allow the experts to place cards between the centroids and the bounds of the cores. Once the decision-makers provide revised cores 
\[
\mathrm{Core}(V_j)=[\,\underline{c}'_j,\overline{c}'_j\,],
\qquad j=1,\dots,k,
\]
we recompute the membership functions to incorporate the corrected structural information. The updated memberships follow this rule:
\begin{itemize}
    \item If $x_i \in \mathrm{Core}(V_j)$ for some $j$, then 
    $V_j(x_i)=1$ and $V_\ell(x_i)=0$ for all $\ell\ne j$.
    \item If $x_i$ is not in any core, then it lies between two cores. Let us denote $j(i)$ the index of the cluster at the left of $x_i$, and  
    apply the two-cluster update rule in Proposition~ \ref{th:update_MFs} using the adjacent core bounds as centroids:
    \[
    V_{j(i)}(x_i)
    =
    \left(
    1+
    \left(
    \frac{(x_i-\overline{c}'_{j(i)})^2}{(x_i-\underline{c}'_{j(i)+1})^2}
    \right)^{\!\frac{1}{m-1}}
    \right)^{-1},
    \qquad
    V_{j(i)+1}(x_i)=1-V_{j(i)}(x_i),
    \]
    with all other memberships equal to zero.
\end{itemize}

The updated cores are then passed to the membership-update algorithm.  Points inside a core receive membership one; points between two adjacent cores have their memberships recomputed using the two-cluster rule. This yields a new fuzzy partition that fully incorporates both the data-driven 
structure and the decision-makers' semantic adjustments.

We remark here that the simplex constraint of a fuzzy partition guarantees that the supports can be directly computed from the adjacent cores. Therefore, in this methodology, determining the core could be equivalently done in terms of the supports. In this sense, if a decision-maker is more comfortable with that, it is possible to carry out this step in terms of the information on the bounds of the supports. Keep in mind that the question for the cores should be \emph{for which values do you have full certainty of belonging to the class?}, whereas the question for the support is \emph{for which range, the values do not belong at all to the class?}.

\begin{example}
Suppose that after the first step of our process, the validated centroids are
$V = (0.20,\; 0.50,\; 0.80)$. After applying the update rule, the C-FKM memberships produce the cores (using $\mathrm{tol}=0.99$)
\[
\mathrm{Core}(V_1)=[0.18,0.22],\qquad
\mathrm{Core}(V_2)=[0.48,0.52],\qquad
\mathrm{Core}(V_3)=[0.78,0.82].
\]
These six numerical values are converted into Deck of Cards units 
(see Theorem \ref{th: num_to_Card}), and the decision-makers examine the resulting cards. Suppose they feel the middle core is  {too narrow} and they modify the number of cards. After reversing the process, we obtain 
\[
\mathrm{Core}(V_2)=[0.46,0.54].
\]
\end{example}

\subsection{Step 3: Fine–Tuning the Left and Right Sides}
\label{subsec:step3}

Once the cores and/or supports of all fuzzy numbers have been validated by the decision-makers, the final refinement stage focuses on adjusting the 
 {shape} of the membership functions outside their cores.  In this step, we analyse the monotone branches of each fuzzy number and identify meaningful ``confidence levels'' that indicate how quickly or slowly the membership should decrease as we move away from the core. These confidence levels are subsequently translated into Deck of Cards units, so that decision-makers can directly manipulate and refine them. We begin by selecting a fuzzy number $V_j$ whose shape we want to refine. Given its validated support and core, $\mathrm{Supp}(V_j)=[s_j^l,s_j^r], \mathrm{Core}(V_j)=[c_j^l,c_j^r]$, we focus separately on its left-hand side and right-hand side:
\[
\text{LHS interval: } [\,s_j^\ell, c_j^\ell\,], 
\qquad
\text{RHS interval: } [\,c_j^r, s_j^r\,].
\]
Suppose that, after visual inspection, the decision-makers decide to adjust 
one of the sides of $V_j$. We therefore extract all data points $x_i$ and corresponding membership values $V_j(x_i)$ lying in either the interval $[s_j^l,c_j^l]$ or $[c_j^r,s_j^r]$. These values form a monotone (increasing or decreasing) sequence of the membership degrees. To analyse and structure the shape of this monotone branch, we cluster the membership values $\{V_j(x_i)\}$ in the corresponding interval using the C-FKM algorithm described in Section~\ref{sec:convex_fkm}. Let $k_{\text{side}}$ denote the number of clusters chosen for the refinement. Applying C-FKM to the sequence $\{V_j(x_i)\}$ yields
\[
\bigl\{c_1 < c_2 < \cdots < c_{k_{\text{side}}}\bigr\},
\]
a strictly increasing set of centroid values, each representing a 
``confidence level'' at which the slope of the membership branch exhibits a change in behaviour. 

These centroids can be interpreted as membership levels at which 
decision-makers may want to introduce semantic distinctions: regions of slight decrease, medium decrease, rapid decrease, and so on. At this stage, these centroids can be adjusted by the decision-maker by the Deck of Cards method using Theorem \ref{th: num_to_Card}.

Note that each membership threshold $c_\ell$ corresponds to exactly one point in the corresponding interval of the fuzzy number because the branch is strictly monotone. Thus, for each adjusted centroid $c_\ell$ we compute the  associated $x$-coordinate by linear interpolation
\[
x_\ell 
\;=\; 
\operatorname{interp}(c_\ell;\, \mu_j(x_i), x_i),
\]
{\it i.e.,}  the unique value of $x$ solving $V_j(x)=c_\ell$.  
This gives the breakpoints
\[
s_j^\ell = x_0 < x_1 < \cdots < x_{k_{\text{side}}}=c_j^\ell
\]
along the corresponding side. The ordered tuple $( x_1,\, x_2,\, \dots,\, x_{k_{\text{side}}})$ is translated into DoC units using Theorem~\ref{th: num_to_Card}.  This yields a sequence of discrete units (cards) representing the distances between consecutive breakpoints, each card corresponding to one elementary ``step'' in the decrease of the membership function.

Once the decision-makers finish adjusting the cards, the modified card 
distribution is mapped back to the $x$-axis using the cumulative-card rule, 
thus producing updated breakpoints $\tilde{x}_1, \dots, \tilde{x}_{k_{\text{side}}}$. The refined membership function on the chosen side is then reconstructed as a piecewise-linear function passing through the points
\[
\bigl(\tilde{x}_1, c_1\bigr),\;
\dots,\;
\bigl(\tilde{x}_{k_{\text{side}}}, c_{k_{\text{side}}}\bigr),
\]
and the resulting membership function remains normal, convex, and supported on the same interval.  
This refinement step provides a final synthesis between the data-driven 
structure produced by C-FKM and the semantic adjustments expressed by the
decision-makers.
\begin{example}
Consider a fuzzy number $V_2$ obtained after Step~2 with a validated core and support
\[
\mathrm{Core}(A)=[0.46,\,0.54], 
\qquad 
\mathrm{Supp}(A)=[0.40,\,0.60].
\]
We focus on refining the left-hand side on the interval $[0.40,0.46]$. Suppose that after inspection, the decision-maker is not satisfied with the values produced by applying C-FKM in Step 2. Then, we apply C-FKM on the left-hand side membership values using $k_{\text{side}}=3$ clusters. This yields three increasing centroid membership levels:
\[
c_1 \approx 0.05,\qquad 
c_2 \approx 0.58,\qquad 
c_3 \approx 0.91.
\]
These centroids represent different ``confidence levels'' on the left-hand side branch. At this stage, the decision-makers can see the equivalent number of cards for these levels, {\it i.e.,}  5 units between $0$ and $c_1$, 52 units between $c_1$ and $c_2$, 34 units between $c_2$ and $c_3$, and 9 units between $c_3$ and $1$. Let us assume that he/she agrees with them and does not modify any card.

Since $V_2$ is strictly increasing on the LHS, each $c_\ell$ corresponds to a unique $x$ value found by linear interpolation. Using the membership of the adjacent values in the dataset, we obtain
\[
x_{c_1}\approx 0.405,\qquad
x_{c_2}\approx 0.441,\qquad
x_{c_3}\approx 0.455.
\]
Thus, the breakpoints are approximately
\[
(0.40,\;0.405,\;0.441,\;0.455,\;0.46).
\]

Using Theorem~\ref{th: num_to_Card} with precision $m=3$ (i.e.\ $N=1000$ cards), we compute the number of cards for showing the decision-maker obtaining, the units 93, 591, 232, and  84, respectively. Then, the decision-maker reviews the card structure and remarks:
\begin{quote}
 {``The difference between the $x$ values of $c_1$ and $c_2$ is too high.
Please make the transition smoother by moving 100 cards from the second 
interval to the first.''}
\end{quote}
Following this instruction, the adjusted card distribution is updated, and the result is retranslated into numerical values using the cumulative-card rule  to new $x$--breakpoints
\[
(\tilde{x}_{1}= 0.4115,\tilde{x}_{2}=0.4410,\tilde{x}_{3}=0.4549),
\]
and the left-hand side of the membership function is reconstructed by piecewise-linear interpolation through the points
\[
(0.40,0),\;
(\tilde x_{1}, c_1),\;
(\tilde x_{2}, c_2),\;
(\tilde x_{3}, c_3),\;
(0.46,1).
\]
\end{example}

\section{Numerical Example}
\label{sec:numerical_example}

To illustrate the proposed hybrid methodology, we apply it to a real dataset
obtained from Kaggle.\footnote{\url{https://www.kaggle.com/datasets/muhammadkhubaibahmad/student-performance-and-clustering-dataset}}
The dataset contains academic performance variables, demographic descriptors, and behavioural indicators related to student learning outcomes. For the present numerical example, we focus exclusively on the variable \texttt{quiz1\_marks}, which records the score obtained by each student in the first continuous assessment quiz. This variable is particularly suitable for our purposes because its distribution captures heterogeneity in student performance: low-performing students accumulate near the lower part of the scale, while a substantial proportion achieves intermediate and high marks. 

The raw data may exhibit missing values and outliers. Therefore, rows containing NaNs in the selected variable must be removed. Figure~\ref{fig:quiz1_hist} displays the resulting histogram together with the associated kernel density estimate (KDE). The shape is mildly right-skewed, with a large concentration of mid-range marks and a 
longer tail extending towards the high-performance region. This distribution 
naturally motivates the construction of five fuzzy classes describing low, medium, and high achievement.


\begin{figure}[ht]
    \centering
    \begin{subfigure}[b]{0.45\textwidth}
        \centering
        \includegraphics[width=\textwidth]{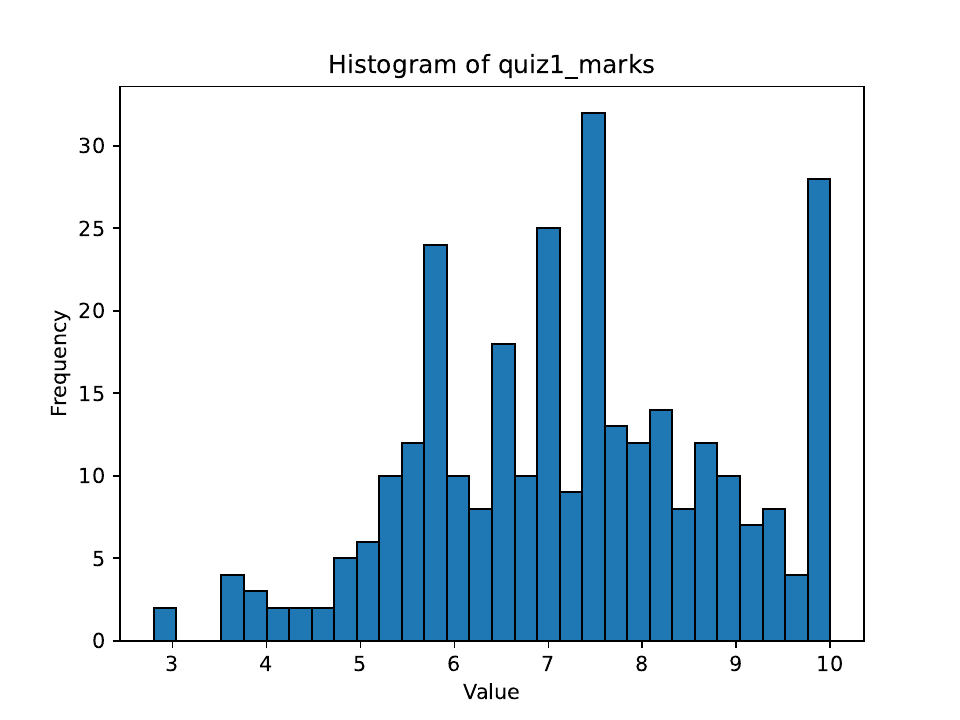}
        \caption{Histogram}
    \end{subfigure}
    \hfill
    \begin{subfigure}[b]{0.45\textwidth}
        \centering
        \includegraphics[width=\textwidth]{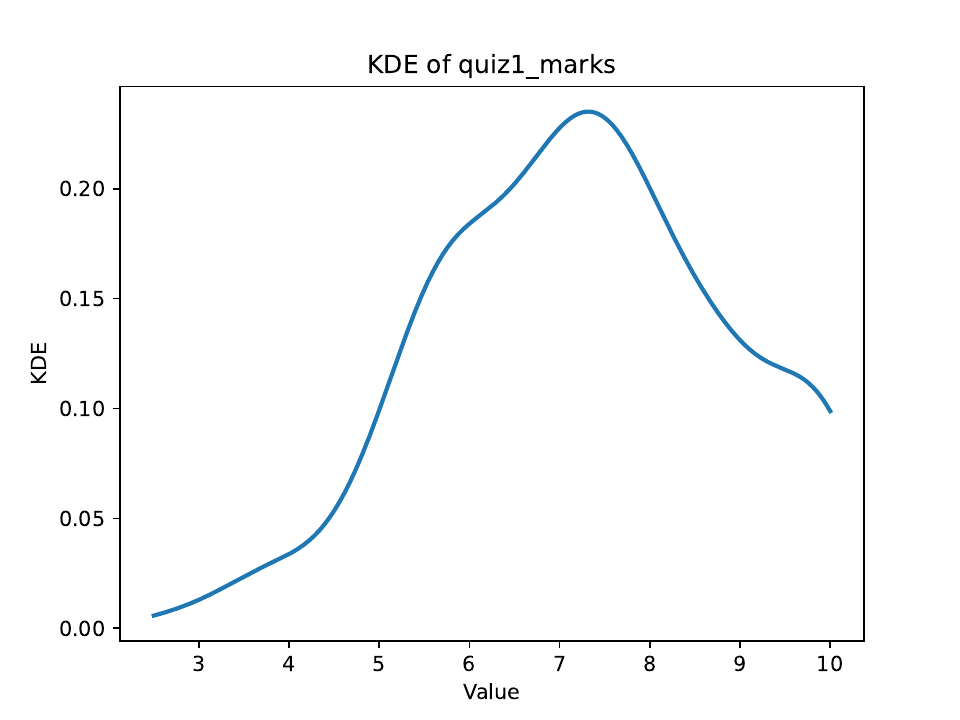}
        \caption{KDE}
    \end{subfigure}
    \caption{Visualization of the data in \texttt{quiz1\_marks}}
    \label{fig:quiz1_hist}
\end{figure}

Firstly, we apply the C-FKM algorithm described in Section~\ref{sec:convex_fkm} with parameters $k=5$ and $m=2$ by initializing the centers as evenly separated values within $[a,b]$, {\it i.e.,}  $v_j=a+(b-a)\frac{j}{k+1}, \ j=1,2,3,4,5$. Let us remark that in our method, each point $x_i$ is allowed to have non-zero membership only in the two consecutive clusters determined by the fixed ordering of the centroids. This guarantees that the resulting membership functions are fuzzy numbers (normal, convex, with compact support) already at the clustering stage. The centroids obtained from running C-FKM on the \texttt{quiz1\_marks} are:
    
\[
v_1 =3.849, \quad
v_2 = 5.683, \quad
v_3 = 7.093 \quad
v_4=8.318\quad
v_5=9.774
\]

These values are the data-driven representative points for the five performance levels. The corresponding membership functions are shown in Figure~\ref{fig:quiz1_MF_step1} together with the centroids. These centroids are transformed to Deck-of-Cards units and shown to the decision-maker for validation. The obtained values are
\begin{equation*}
    a =2.8\ [14]\ v_1\ [26]\ v_2\ [19]\ v_3\ [17]\ v_4\ [20]\ v_5\ [4]\ b=10
\end{equation*}
where the upper and lower bounds $[a,b]$ have been set as the bounds of the dataset. In view of this, the decision-maker considers that the highest performance level is too high, and thus he moves 5 cards from between $v_4$ and $v_5$ to the range between $v_5$ and $b$. After the computations, the resulting centroid vector is $v=(3.808, 5.68,  7.048, 8.272, 9.352)$

\begin{figure}[ht]
    \centering
    \begin{subfigure}[b]{0.45\textwidth}
        \centering
        \includegraphics[width=\textwidth]{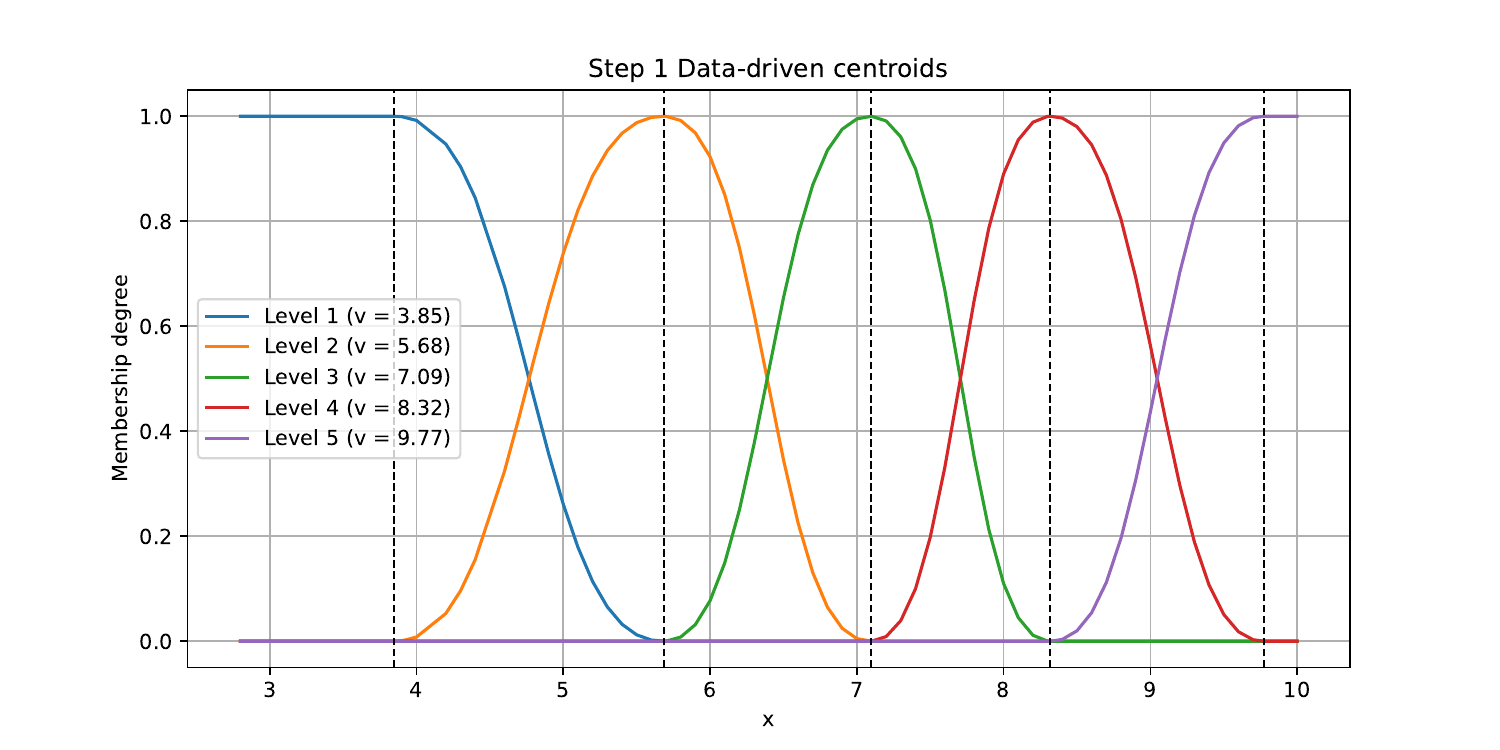}
        \caption{Data-driven}
    \end{subfigure}
    \hfill
    \begin{subfigure}[b]{0.45\textwidth}
        \centering
        \includegraphics[width=\textwidth]{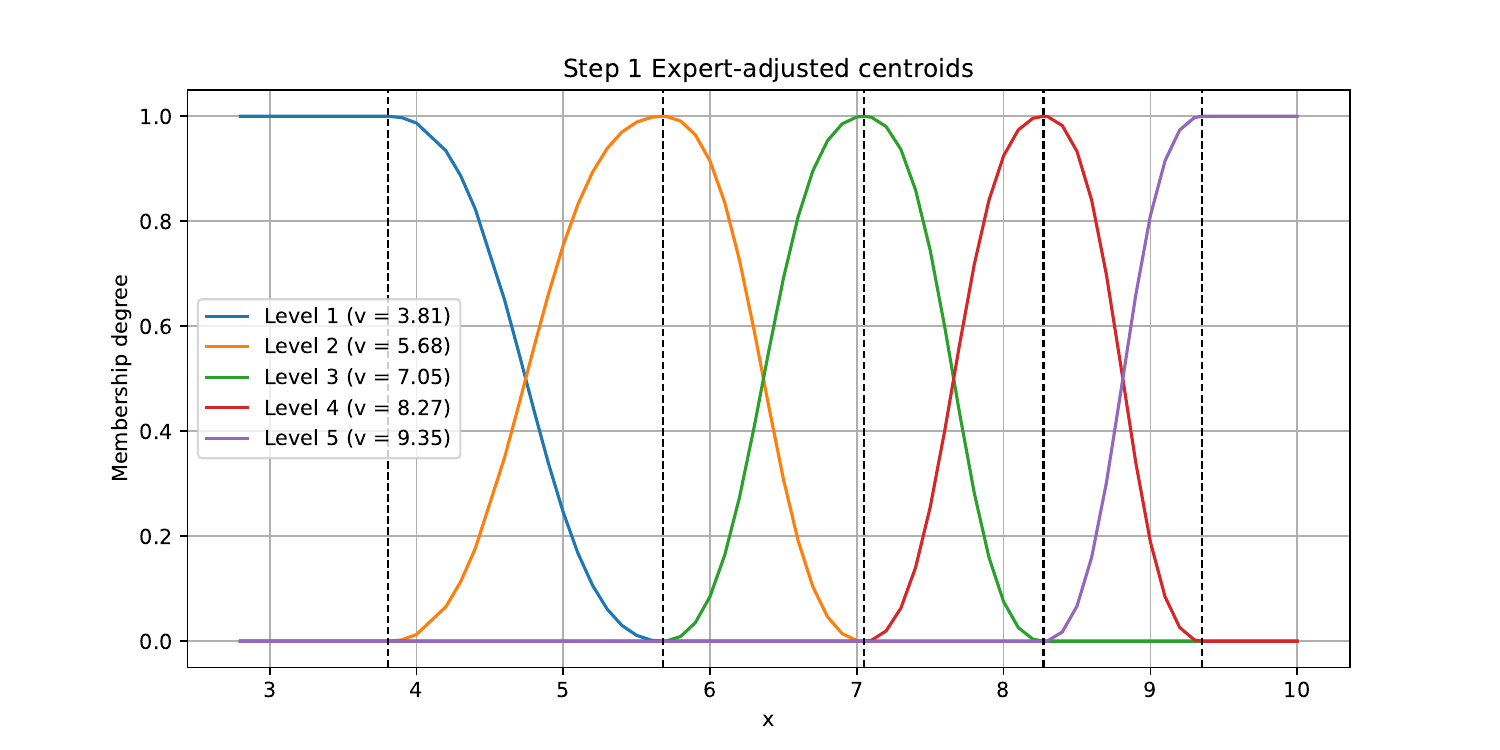}
        \caption{Expert-adjusted}
    \end{subfigure}
    \caption{Membership functions obtained from C-FKM}
    \label{fig:quiz1_MF_step1}
\end{figure}

Given the C-FKM membership functions, the core of each fuzzy number is defined 
as the interval of values in the dataset where $V_j(x) \geq 1-\tau$, with $\tau=0.01$. Supports are defined as the minimal intervals on which $\mu_j(x)>0$, and can be computed directly from the cores due to the model constraints. Consequently, the cores are given by $[2.8, 3.9], [5.6, 5.8], [7.0, 7.1], [8.2, 8.3]$ and $[9.3, 10.0]$, and the obtained chain of lower and upper bounds, {\it i.e.,}  $2.8,  3.9,  5.6,  5.8,  7,   7.1,  8.2,  8.3,  9.3, 10$ is transformed into units using again Theorem \ref{th: num_to_Card}, and presented to the decision-maker for adjustment:
\begin{equation*}
    c_1^-\, [15]\, c_1^+ [23]\, c_2^- [3] \,c_2^+ [17]\, c_3^-\, [1]\, c_3^+ \,[15] \,c_4^-\,  [2]\, c_4^+\, [14]\, c_5^-\, [10] \,c_5^+
\end{equation*}

After checking the values, he decision-maker decides to update the units as follows:
\begin{equation*}
   c_1^-\, [14]\, c_1^+ [19]\, c_2^- [7] \,c_2^+ [14]\, c_3^-\, [5]\, c_3^+ \,[12] \,c_4^-\,  [5]\, c_4^+\, [14]\, c_5^-\, [10] \,c_5^+
\end{equation*}
resulting in the following chain for the validated cores:
\begin{equation*}
     [2.8,    3.808]\,  [5.176,  5.68]\,   [6.688 , 7.048]\,  [7.912,  8.272]\,  [9.28,  10]
\end{equation*}
In Figure \ref{fig:quiz1_MF_step2}, we display the comparison between the cores and supports before and after expert intervention. Note that the validated centroids lie within their respective cores.

\begin{figure}[ht]
    \centering
    \begin{subfigure}[b]{0.45\textwidth}
        \centering
        \includegraphics[width=\textwidth]{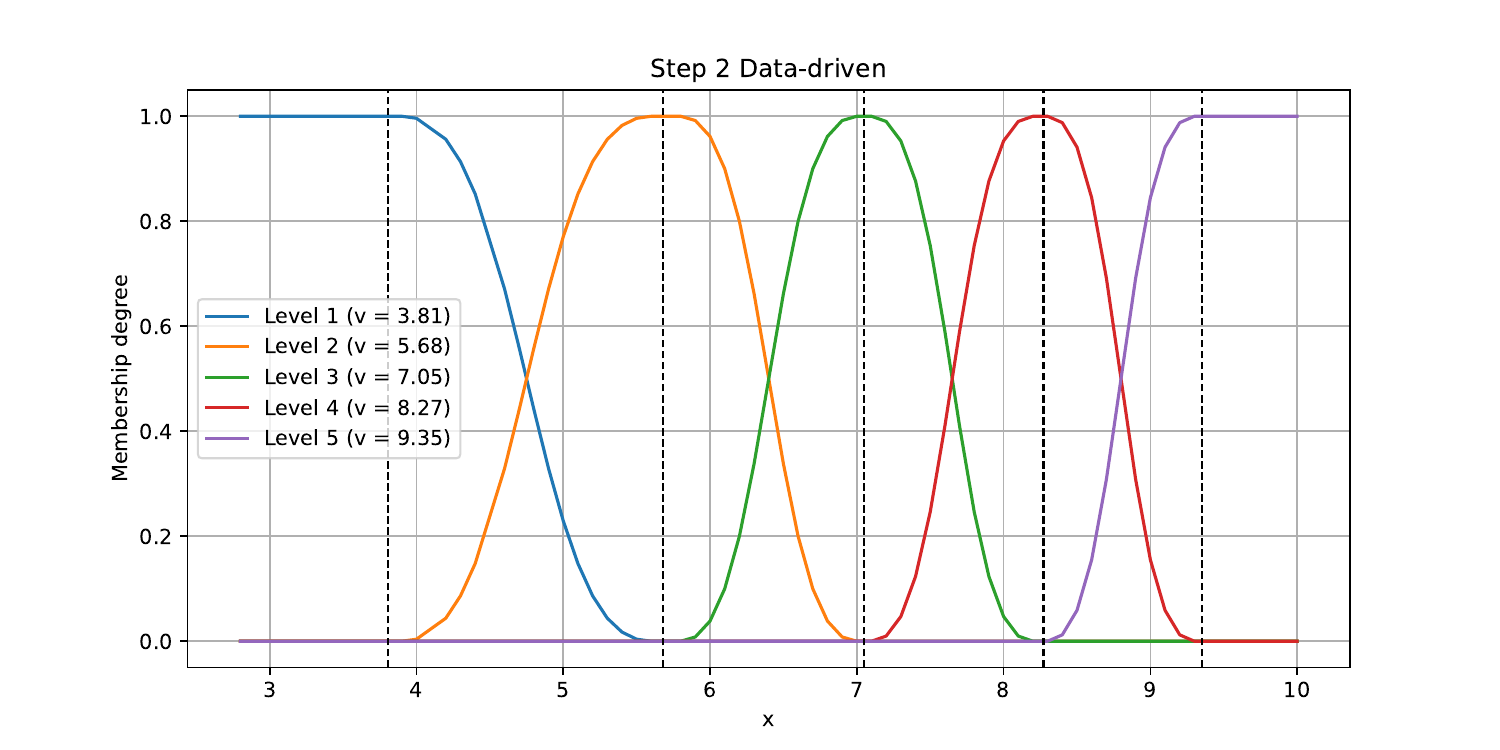}
        \caption{Data-driven}
    \end{subfigure}
    \hfill
    \begin{subfigure}[b]{0.45\textwidth}
        \centering
        \includegraphics[width=\textwidth]{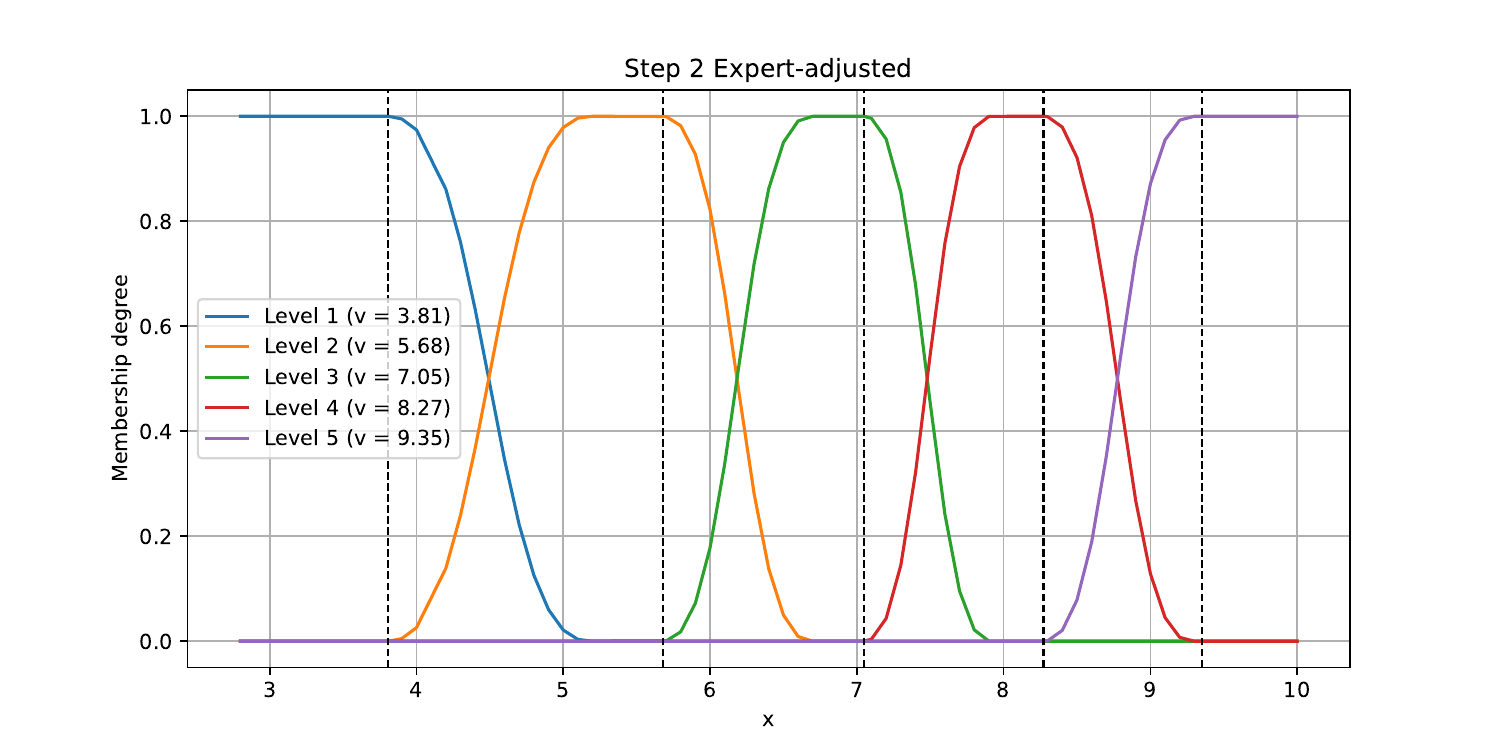}
        \caption{Expert-driven}
    \end{subfigure}
    \caption{Membership functions after validating cores and supports}
    \label{fig:quiz1_MF_step2}
\end{figure}

After we show the new version of the membership functions to the decision-maker, he confirms that he is satisfied with all the memberships, with the exception of the behaviour between the first and the second cores.

For the values on this side, the membership degrees of $V_0$ at the right of the core are clustered to identify the three representative points, namely $0.08,0.6$ and $0.93$ (see Figure \ref{fig:quiz1_MF_step3}). The decision-maker says that these membership degrees are representative enough, so we proceed to interpolate and find the corresponding values in the $x$ axis (i.e, the upper bounds of the corresponding $\alpha-$cut), which are, respectively, $4.86, 4.58$, and $4.06$. The decision-maker feels that these values do not represent the right break-points for the membership. Therefore, we express them as cards, and after modifying the units between the levels, we finally obtain the validated breakpoints $4.06$, $4.46$, and $4.78$. After this step, the final memberships are obtained, as displayed in Figure \ref{fig:quiz1_MF_step3}.
\begin{figure}[ht]
    \centering
    \begin{subfigure}[b]{0.45\textwidth}
        \centering
        \includegraphics[width=\textwidth]{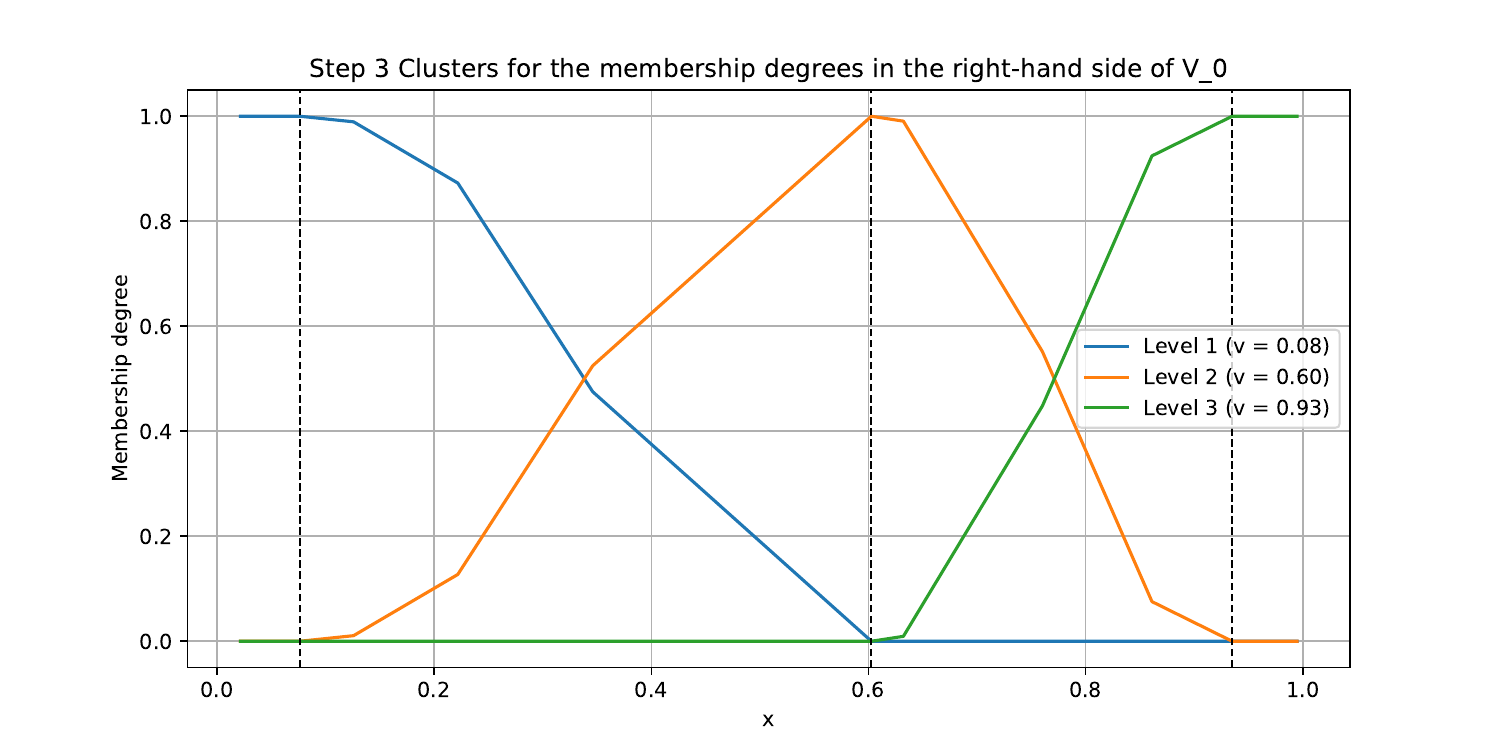}
        \caption{Clusters of the membership degrees}
    \end{subfigure}
    \hfill
    \begin{subfigure}[b]{0.45\textwidth}
        \centering
        \includegraphics[width=\textwidth]{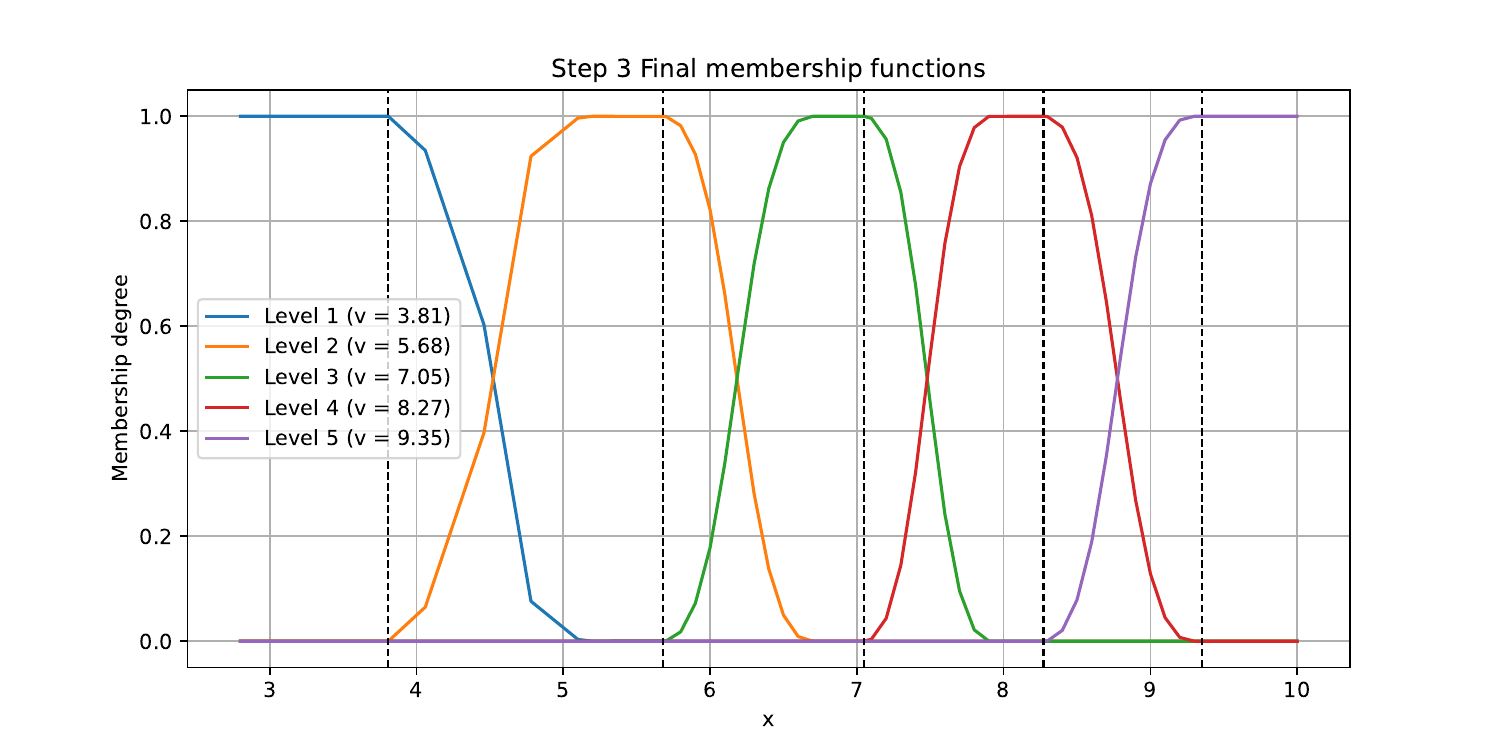}
        \caption{Final membership functions}
    \end{subfigure}
    \caption{Membership functions obtained at the end of the process}
    \label{fig:quiz1_MF_step3}
\end{figure}


\section{Influence of data distribution on fuzzy uncertainty profiles}
\label{sec:comparison}

Figures~\ref{fig:dist_histogram}, \ref{fig:dist_3}, and \ref{fig:dist_5} present a comparison of the 
output of the methodology when applied to three synthetic datasets with distinct distributional properties: a symmetric dataset, a skewed dataset, and 
a multimodal dataset. Since the data distributions differ from normality, in this section, we have considered the initialized centroids as the percentiles $100\cdot j/(k+1)$, for $j=1,...,k$. This comparative analysis aims to assess how the hybrid approach adapts to different uncertainty structures present in the data and to evaluate whether the resulting fuzzy numbers appropriately reflect such variations. 

Figure~\ref{fig:dist_histogram} shows that the first dataset is well represented by a bell-shaped distribution centred near the midpoint of the scale, whereas the second exhibits a clear right skew with most values concentrated in the lower region. The third dataset displays three separated modes, indicating 
multiple latent subgroups in the underlying population. These qualitative features constitute the basis upon which the fuzzy numbers must be constructed.

When the number of fuzzy classes is set to $k=3$ (Figure~\ref{fig:dist_3}), the membership functions resulting from the symmetric dataset present balanced 
shapes with smooth transitions and almost equidistant centroids. This alignment reflects a situation in which uncertainty is distributed uniformly around a central concept, and no region of the domain dominates in frequency or relevance. In the skewed case, however, the shapes become clearly asymmetric: the lowest class acquires a narrow core, closely following the concentration of data, while the upper class 
spreads its support widely over a region with limited evidence. This behaviour correctly captures the notion that high values are rare and should therefore 
be associated with higher uncertainty. For the multimodal dataset, the method identifies distinct transition areas that coincide with local density valleys, thus revealing a structure that a unimodal model would obscure.

Increasing the number of classes to $k=5$ (Figure~\ref{fig:dist_5}) enhances the granularity of the representation and further highlights the influence of distributional characteristics. In the symmetric dataset, the additional fuzzy numbers appear orderly and evenly spaced, strengthening interpretability 
without altering the inherent symmetry of the partition. In contrast, for the skewed and multimodal datasets, the additional classes tend to populate the densest regions of the data while leaving wide, low-evidence intervals in the tails or between modal clusters. The shapes become more complex but remain meaningfully related to the underlying observations: abrupt changes in membership occur where the data suggest clear boundaries, whereas slow transitions emerge in more ambiguous or underrepresented regions.

\begin{figure}[H]
    \centering
    \includegraphics[width=0.3\textwidth]{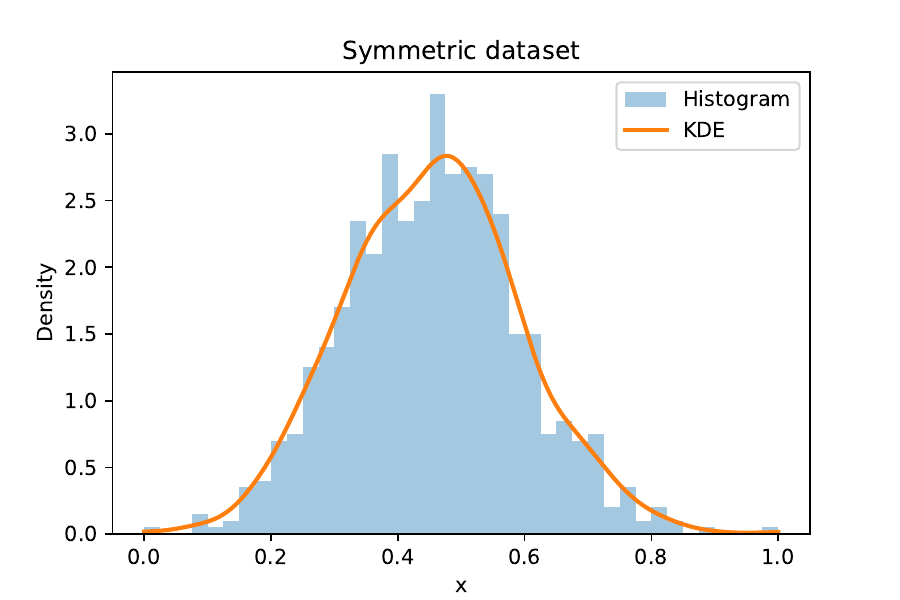}
    \includegraphics[width=0.3\textwidth]{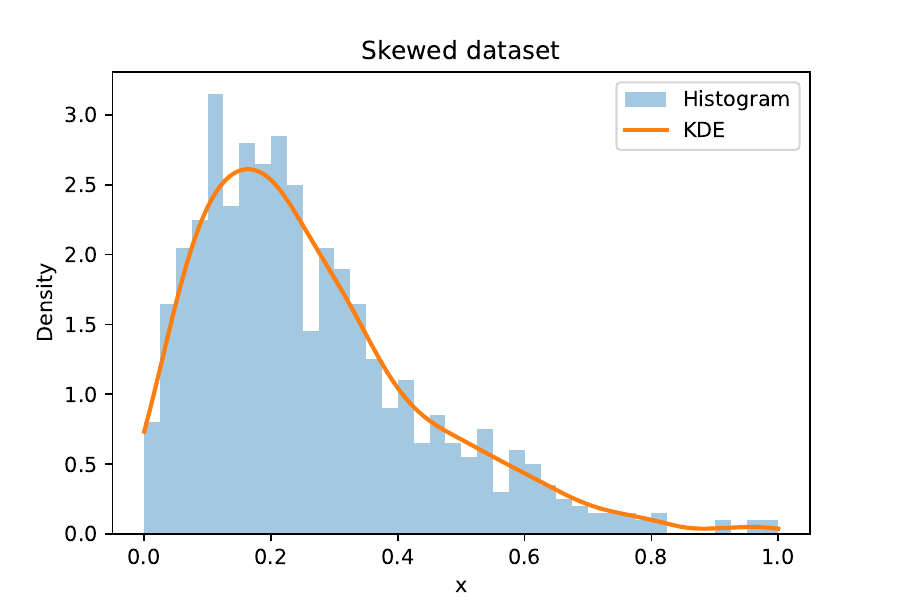}
    \includegraphics[width=0.3\textwidth]{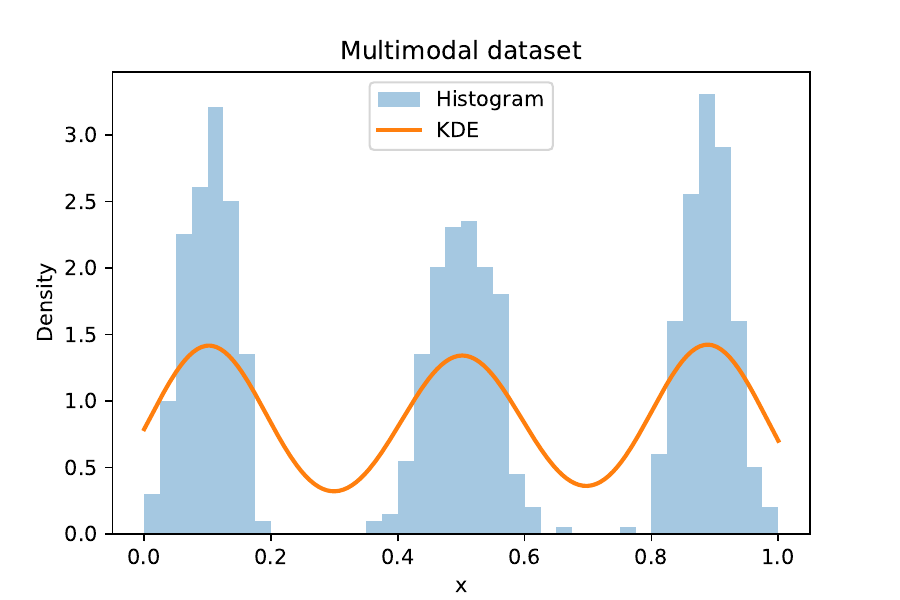}

    \caption{Histograms and KDE of the data}
    \label{fig:dist_histogram}
\end{figure}

\begin{figure}[H]
    \centering
    \includegraphics[width=0.3\textwidth]{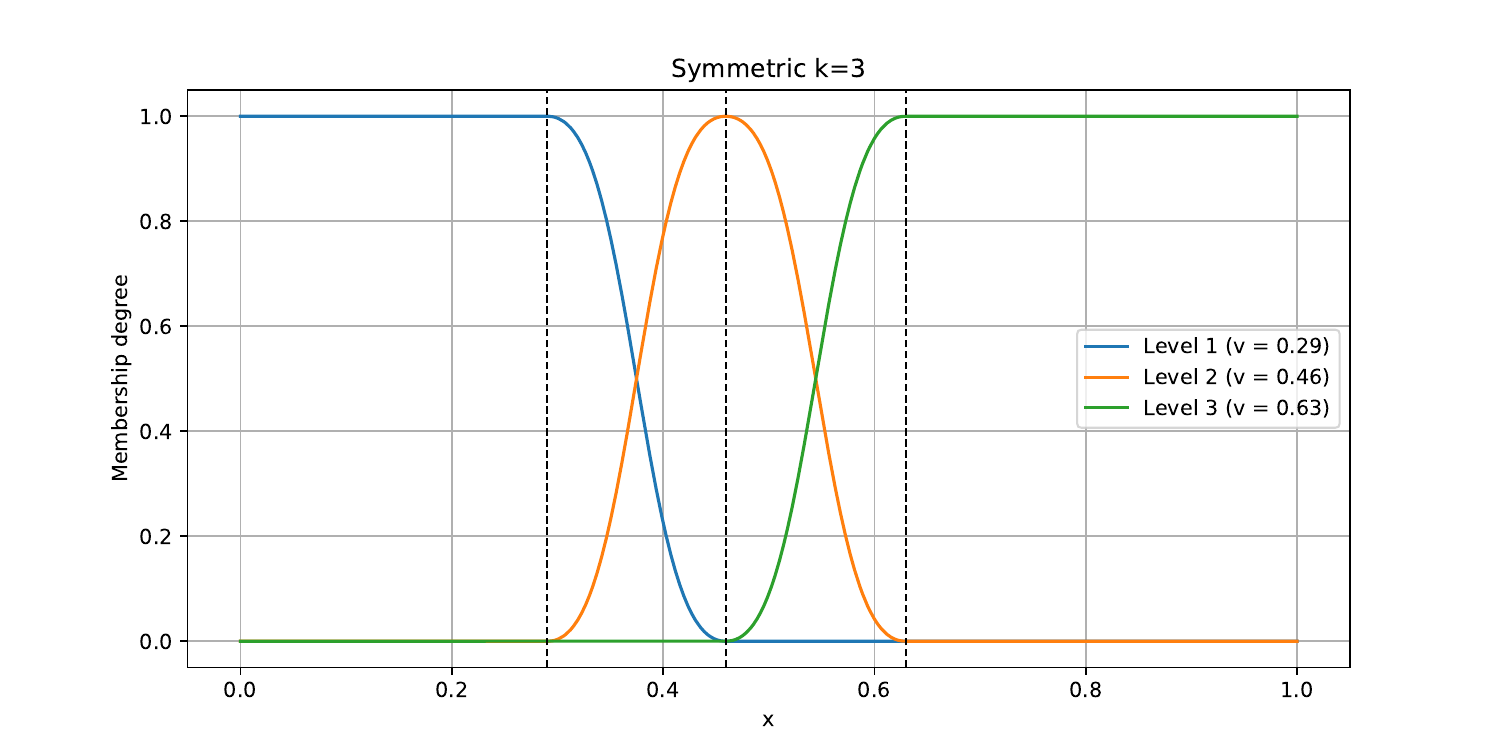}        \includegraphics[width=0.3\textwidth]{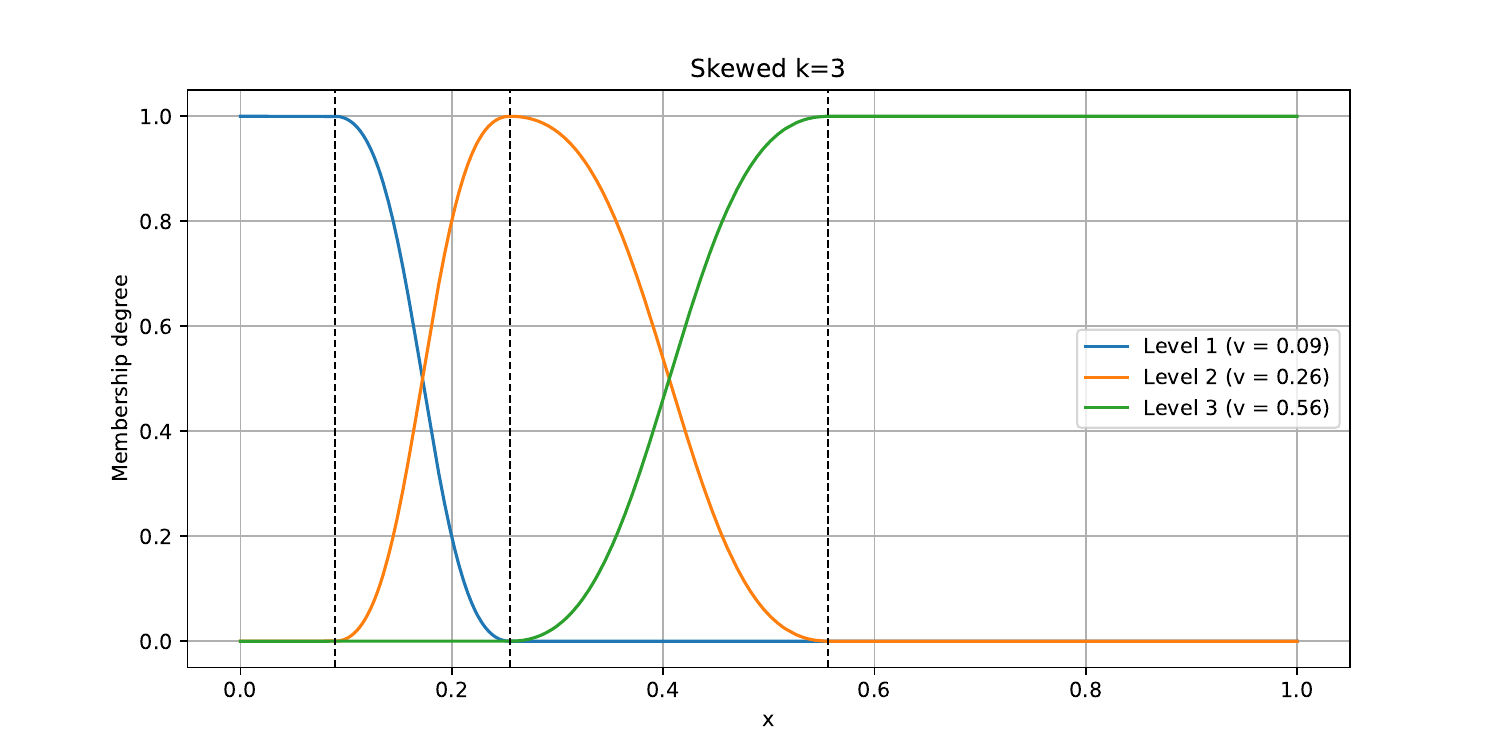}        \includegraphics[width=0.3\textwidth]{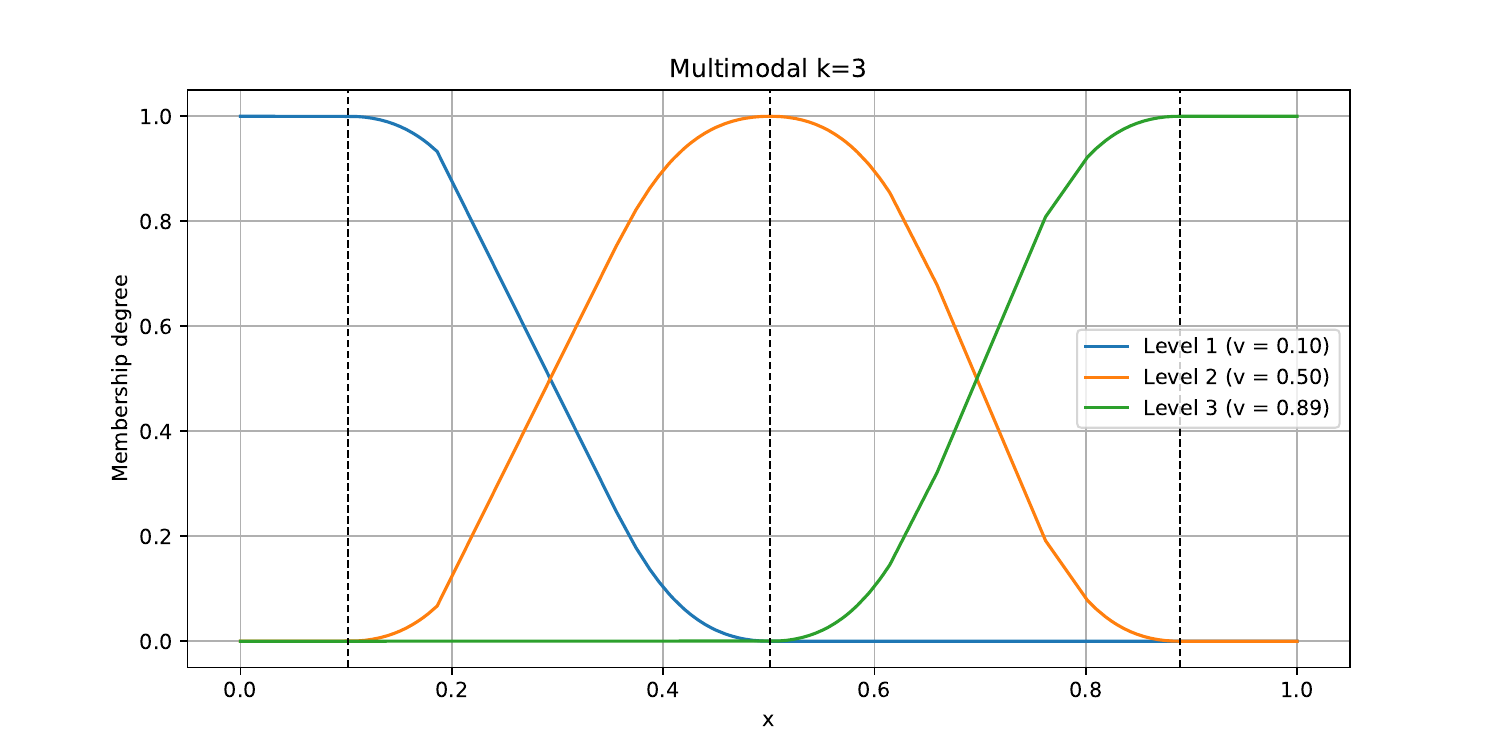}
    \caption{Fuzzy classes for $k=3$}
    \label{fig:dist_3}
\end{figure}

\begin{figure}[H]
    \centering        \includegraphics[width=0.3\textwidth]{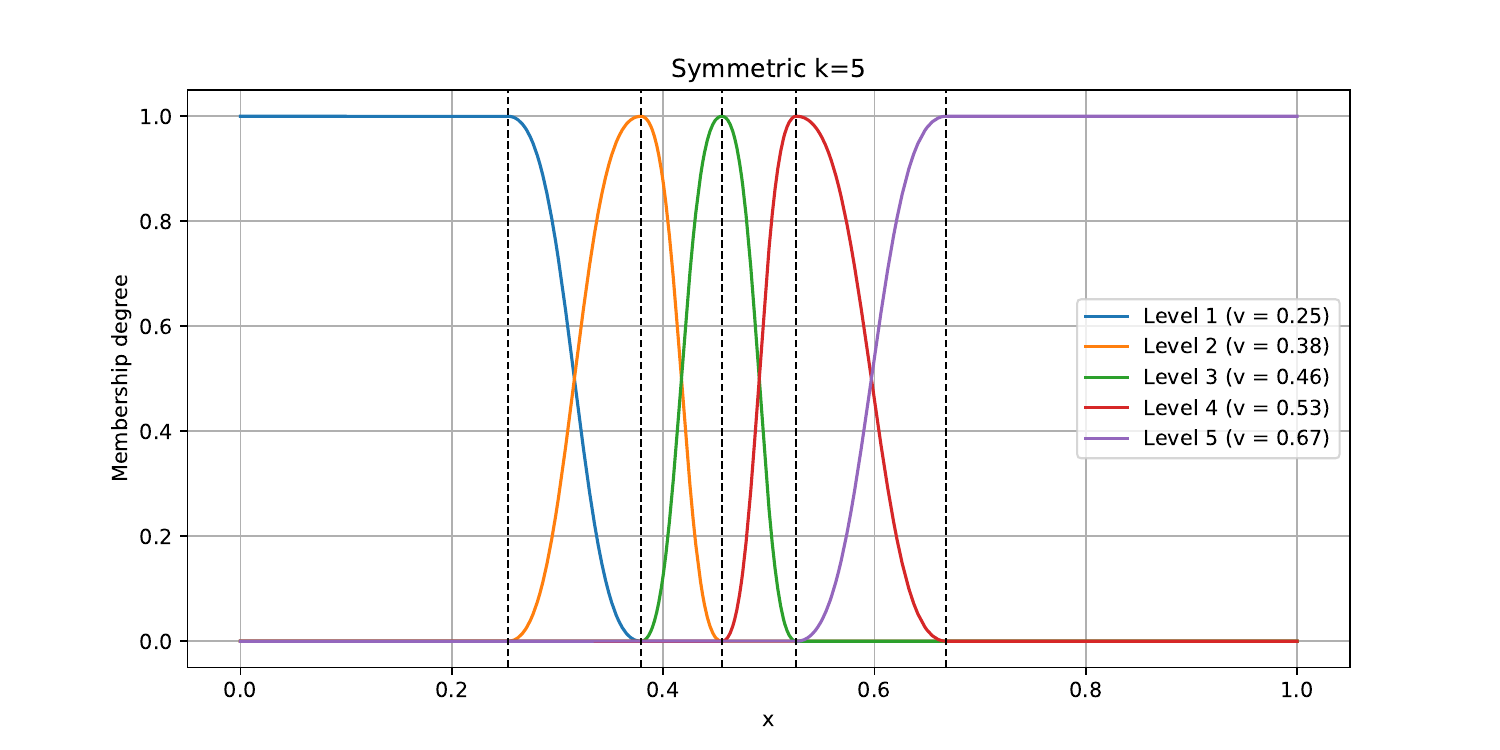}
    \includegraphics[width=0.3\textwidth]{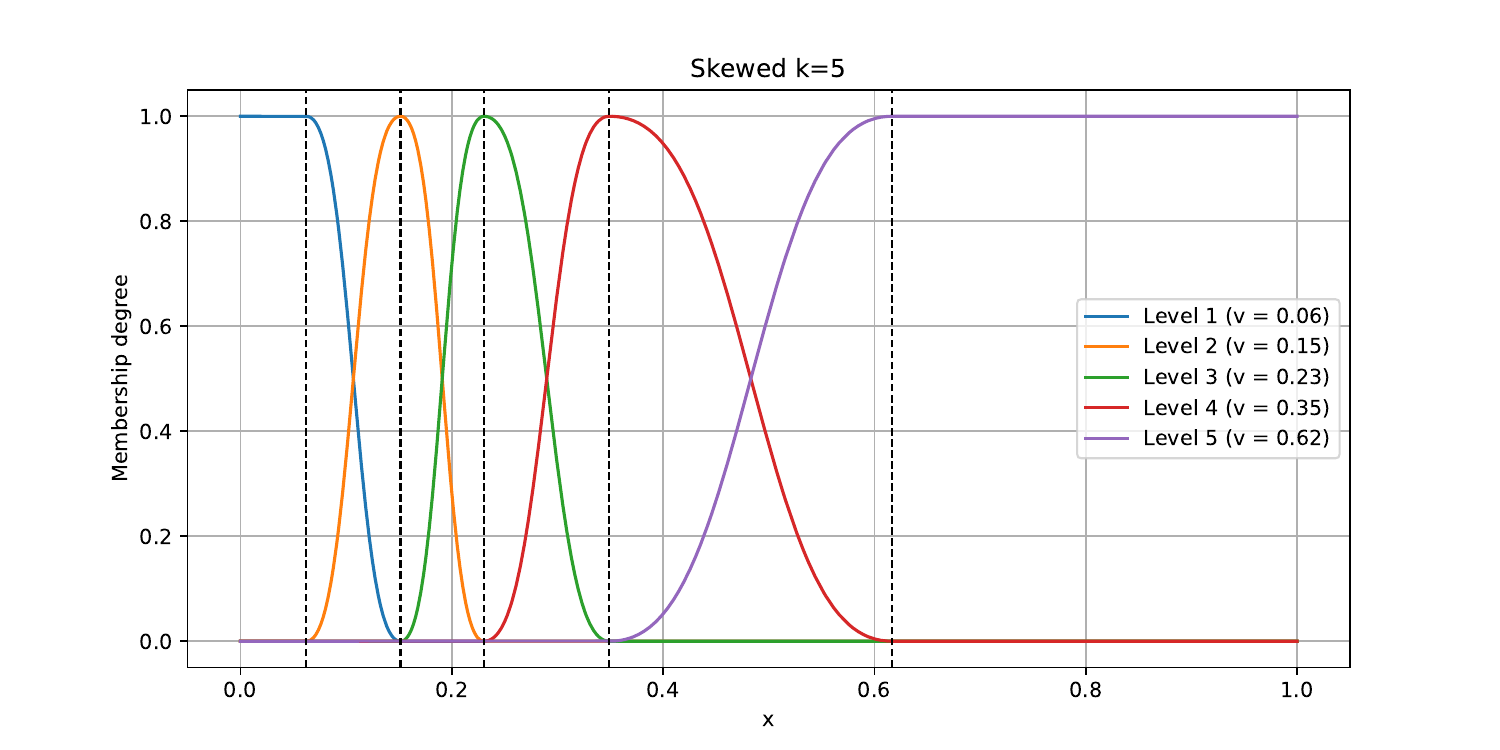}
    \includegraphics[width=0.3\textwidth]{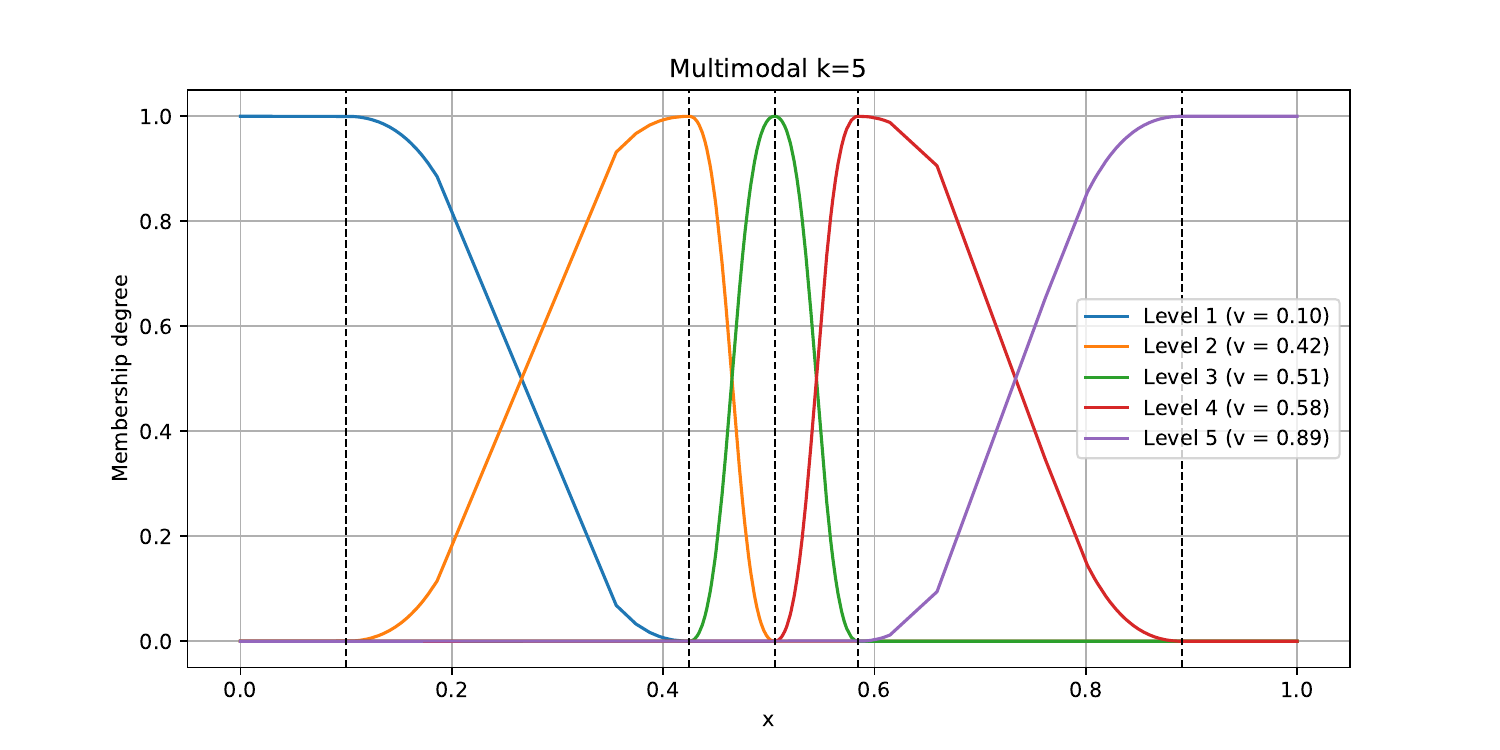}
    
    \caption{Fuzzy classes for $k=5$}
    \label{fig:dist_5}
\end{figure}
\section{Conclusions}
\label{sec:conclusion}

This work has introduced a hybrid socio-technical methodology for constructing fuzzy numbers by
combining numerical data with expert elicitation through the DoC method. The proposed approach extends the DoC-MF framework by incorporating a data-driven pipeline based
on a novel convex version of the classical fuzzy $k$-means clustering algorithm, which ensures that each fuzzy set generated during the computational stage is a fuzzy number by construction. At each step, the numerical outputs are transformed into card-based units, enabling domain experts to validate and refine the shapes, cores, supports, and slopes of the resulting membership functions using a transparent and cognitively meaningful representation.

The synthetic case studies demonstrated that the methodology adapts consistently to different distributional structures such as symmetry, skewness, and multimodality, without imposing rigid functional forms or sacrificing interpretability. The numerical experiment with real educational data further illustrated the ability of the approach to capture heterogeneous performance profiles, while still preserving the semantic meaning required for decision support.

The methodology provides an effective compromise between purely data-driven modeling, which may lack interpretability, and purely expert-driven construction, which may fail to reflect empirical evidence. The resulting membership functions faithfully represent both quantitative information and qualitative judgements, offering a robust foundation for downstream decision-making processes.

The current work opens new venues for future research. Firstly, we will explore the integration of alternative distance measures into C-FKM to better handle ordinal, categorical, or mixed data types. Secondly, the third step of the methodology could be further enhanced by 
extending C-FKM to multidimensional settings, enabling simultaneous refinement of both sides of a membership
function and capturing more complex uncertainty shapes. Finally, we aim to integrate the method into a complete MCDA framework, evaluating its impact on ranking robustness and preference learning tasks in real-world decision problems.

\section*{Disclosure of interest}
The authors report there are no competing interests to declare.

\section*{Declaration of Generative AI and AI-assisted technologies in the writing process}
During the preparation of this work, the authors used Gemini 3 (Google) and Grammarly (both accessed in December 2025) to improve the language of the manuscript. After using these tools, the authors reviewed and edited the content as needed and take full responsibility for the content of the publication.
\section*{Acknowledgments}
\noindent José Rui Figueira is financed by Portuguese funds through the FCT – Foundation for Science and Technology under project UID/97/2025 (CEGIST). Diego García-Zamora is financed by the mobility grant CAS24/00249 from the Spanish Ministry of Science, Innovation, and Universities, which supported his research stay at the Instituto Superior Técnico, Universidad de Lisboa. Diego García-Zamora also acknowledges the support of CEGIST for all the assistance during his research stay.

\addcontentsline{toc}{section}{\numberline{}References}
\bibliographystyle{model5-names}
\bibliography{biblio}

\end{document}